\documentclass{amsart}

\usepackage[utf8]{inputenc}
\usepackage{amsmath,amssymb,bbm,enumitem}
\usepackage{hyperref}
\usepackage{amsthm}
\usepackage[capitalize]{cleveref}
\usepackage[square,numbers,sort]{natbib}

\hypersetup{pdfstartview=}

\newcommand\BBC{\mathbb C}

\newcommand\BZ{\mathbb Z}

\newcommand{\cyc}[1]{\langle #1 \rangle}

\newcommand\Img{\operatorname{Im}}

\def\BHM#1.#2.#3.#4.{{^{#1}_{#3}\mathcal B^{#2}_{#4}}}

\newcommand\comm\curlyvee
\newcommand\cocomm\curlywedge

\DeclareMathOperator{\Hom}{Hom}
\DeclareMathOperator{\id}{id}
\DeclareMathOperator{\Aut}{Aut}
\DeclareMathOperator{\End}{End}
\DeclareMathOperator{\Inn}{Inn}
\DeclareMathOperator{\Out}{Out}
\DeclareMathOperator{\BrPic}{BrPic}
\DeclareMathOperator{\GVec}{Vec}

\theoremstyle{plain}
\newtheorem{thm}{Theorem}[section]

\newtheorem{cor}[thm]{Corollary}
\newtheorem{prop}[thm]{Proposition}
\newtheorem{lem}[thm]{Lemma}

\theoremstyle{definition}
\newtheorem{df}[thm]{Definition}
\newtheorem{example}[thm]{Example}
\newtheorem{question}[thm]{Question}

\theoremstyle{remark}
\newtheorem{rem}[thm]{Remark}

\crefname{lem}{Lemma}{Lemmas}
\crefname{thm}{Theorem}{Theorems}
\crefname{cor}{Corollary}{Corollaries}
\crefname{prop}{Proposition}{Propositions}
\crefname{example}{example}{examples}
\crefname{df}{Definition}{Definitions}
\crefname{equation}{equation}{equations}

\numberwithin{equation}{section}

\renewcommand\iff{\Leftrightarrow}

\def\clap#1{\hbox to 0pt{\hss#1\hss}}

\def\mathrlap{\mathpalette\mathrlapinternal}
\def\mathclap{\mathpalette\mathclapinternal}

\def\mathrlapinternal#1#2{%
\rlap{$\mathsurround=0pt#1{#2}$}}
\def\mathclapinternal#1#2{%
\clap{$\mathsurround=0pt#1{#2}$}}

\def\D{\mathcal D}
\def\morphquad{(p,u,r,v)}                            
\newcommand{\BCh}[1]{\Hom(#1,\widehat{#1})}
\DeclareMathOperator{\SpAutc}{SpAut_c}               
\newcommand{\ucom}[1]{^{(#1)}}
\newcommand{\com}[1]{_{(#1)}}
\newcommand\ot{\otimes}
\newcommand\inv{^{-1}}
\newcommand{\kk}{\mathbbm{k}}
\newcommand{\du}[1]{\kk^{#1}}
\DeclareMathOperator{\Rep}{Rep}
\DeclareMathOperator{\ev}{ev}
\newcommand{\genhom}{\Hom_{\operatorname{QB}}(\D^\omega(G),\D^\eta(H))}
\newcommand{\genhomh}{\Hom(\D^\omega(G),\D^\eta(H))}
\newcommand{\genhomg}{\Hom_{\operatorname{QB}}(\D^\omega(G),\D^\eta(G))}
\DeclareMathOperator{\Isom}{Isom}
\newcommand{\genisomr}{\Isom_{\operatorname{rigid}}(\D^\omega(G),\D^\eta(H))}
\newcommand{\genisomrg}{\Isom_{\operatorname{rigid}}(\D^\omega(G),\D^\eta(G))}
\newcommand{\genautr}{\Aut_{\operatorname{rigid}}(\D^\omega(G))}
\DeclareMathOperator{\lact}{\rightharpoonup}

\newlist{propenum}{enumerate}{1} 
\setlist[propenum]{label=\roman*), ref=\textup{\thethm~(\roman*)}}
\crefalias{propenumi}{proposition}

\newlist{lemenum}{enumerate}{1}
\setlist[lemenum]{label=\roman*), ref=\textup{\thethm~(\roman*)}}
\crefalias{lemenumi}{lemma}

\newlist{thmenum}{enumerate}{1}
\setlist[thmenum]{label=\roman*), ref=\textup{\thethm~(\roman*)}}
\crefalias{thmenumi}{theorem}

\title[Rigid isomorphisms of twisted group doubles]{Homomorphisms and rigid isomorphisms of twisted group doubles}
\author{Marc Keilberg}
\email{keilberg@usc.edu}
\begin{document}
\begin{abstract}
We prove several results concerning quasi-bialgebra morphisms \ensuremath{\D^\omega(G)\to\D^\eta(H)} of twisted group doubles.  We take a particular focus on the isomorphisms which are simultaneously isomorphisms \ensuremath{\D(G)\to\D(H)}.  All such isomorphisms are shown to be morphisms of quasi-Hopf algebras, and a classification of all such isomorphisms is determined.  Whenever \ensuremath{\omega\in Z^3(G/Z(G),U(1))} this suffices to completely describe \ensuremath{\Aut(\D^\omega(G))}, the group of quasi-Hopf algebra isomorphisms of \ensuremath{\D^\omega(G)}, and so generalizes existing descriptions for the case where \ensuremath{\omega} is trivial.
\end{abstract}
\maketitle
\tableofcontents
\newpage

\section*{Introduction}
The goal of this paper is to generalize existing results on the Hopf algebra automorphisms of $\D(G)$ \citep{K14,KS14,LenPri:Monoidal} to the case of twisted doubles, as introduced by \citet{DPR}.

A twisted double $\D^\omega(G)$, where $\omega\in Z^3(G,U(1))$ is a 3-cocycle of the finite group $G$ with values in the circle group $U(1)$, forms a quasitriangular quasi-Hopf algebra that characterizes lower dimensional Dijkgraaf-Witten theories with finite gauge group \citep{HuWanYu:TQD,DijWit:TopGauge,FQ:ChernFinite,Morton2013}.  These objects, especially their representation categories, have been of exceptional recent interest in mathematical physics.

By \citep{ENO2} there are group isomorphisms
\[\Aut_{\text{br}}(\Rep(D^\omega(G)))\cong \Aut_{\text{br}}(Z(\GVec_G^\omega))\cong \BrPic(\GVec_G^\omega),\]
where $\GVec_G^\omega$ is the category of $G$-graded $\BBC$-vector spaces with associator determined by $\omega$.  The Brauer-Picard group $\BrPic(\mathcal C)$ of a finite tensor category $\mathcal{C}$ appears in the classification problem of $G$-extensions of fusion categories \citep{ENO2}.  In mathematical physics, the Brauer-Picard group appears as the group of symmetries of certain TQFTs \citep{FSV,FPSV:AbBrauer}.  An explicit description of $\Aut_{\text{br}}(\Rep(\D(G)))$ is already an interesting and difficult problem, and is of significant and ongoing interest.  \citet{NR14} gave a procedure that can, in principle, allow the computation of $\Aut_{\text{br}}(\Rep(\D(G)))$ for a given group $G$, but this remains an ad hoc procedure.  More recently, and in a similar vein, \citet{MarNik:BPofFus} introduced several other methods for studying the Brauer-Picard group of $\GVec_G^\omega$, and applied their methods to completely describe the Brauer-Picard groups associated to certain classes of fusion categories of prime power dimension.  The case with $G$ abelian and trivial 3-cocycle was considered in \citep{FPSV:AbBrauer}, and the special case of the (unique) non-abelian group of order $p^3$ and exponent $p$ was handled by \citet{Rie}.  \citet{LenPri:Brauer} conjecture that $\Aut_{\text{br}}(\D(G))$ can be determined by adding certain cohomological data to monoidal autoequivalences naturally obtained from $\Aut(\D(G))$.  Namely, it is determined when the tensor structure for the autoequivalence naturally defined by an element of $\Aut(\D(G))$ can be altered by lazy cohomology \citep{BicCar:Lazy} data to produce a braided autoequivalence.  The approach is an analogue to a Bruhat decomposition.  The conjecture was verified for the few examples that were fully worked out in \citep{NR14}, but otherwise remains open.  It is in this light that we pursue a description of $\Aut(\D^\omega(G))$.

The paper is organized as follows.  In \cref{sec:prelim} we enumerate the preliminary definitions, results, and notations the paper will use.  In \cref{sec:components} we provide the main tool for our investigation, which is to decompose quasi-bialgebra morphisms $\D^\omega(G)\to\D^\eta(H)$ into a quadruple of components $\morphquad$.  This is analogous to the decompositions given by \citet{ABM} for bicrossed product Hopf algebras, and reproduces the decompositions they obtain for the special case of morphisms $\D(G)\to\D(H)$.  Extracting useful information from this decomposition is a laborious task, and in \cref{sec:fundamental,sec:p-embeddings} we establish the fundamental properties for the components that undergird all other calculations in the paper.  The component $p$ is of particular importance, and the most well-behaved, and so receives a great deal of attention.  \cref{sec:coassoc,sec:bialg-ids} then proceeds to enumerate the identities that must be satisfied to determine a morphism of quasi-bialgebras $\D^\omega(G)\to\D^\eta(H)$. This places several constraints on how the 3-cocycles relate to the components $p$ and $u$ in particular.  While $p$ is easily the nicest of the four components, the components $u,v$ are also reasonably well-behaved.  \cref{sec:uv-properties,sec:uv-bialg} investigate these components in detail. We are able to extract a fair amount of information regarding their properties and how they are constrained by, and related to, the structures of $\D^\omega(G)$ and $\D^\eta(H)$.  The most ill-behaved component is $r$, and it is not expected that this component admits any simpler description in general than the trivial one: $\morphquad$ is a quasi-bialgebra morphism.  \cref{sec:r-coalg,sec:r-alg} are subsequently dedicated towards investigating $r$ under certain hypotheses on the other three components, and demonstrating how ill-behaved it becomes as we relax these assumptions.  This directs us towards the ultimate goal of the paper, considered in \cref{sec:antipode,sec:rigid-isoms,sec:rigid-auts}, which is to classify those quasi-bialgebra isomorphisms which also naturally determine an isomorphism of Hopf algebras $\D(G)\to\D(H)$.  The latter case has been thoroughly explored in \citep{K14,KS14,LenPri:Monoidal}.  We show that not only do such isomorphisms exist---with several examples and a counterexample detailed in \cref{sec:rigid-examples}---, but that there are even automorphisms of $\D^\omega(G)$ which are defined independently of $\omega$ whenever $G$ is not perfect.  We consider a modification of this idea that uses only cohomology classes, and pose a few questions concerning 3-cocycles of $G$ that are related to the possible existence of other examples of such universally defined isomorphisms. The results we obtain completely describe $\Aut(\D^\omega(G))$ whenever $\omega\in Z^3(G/Z(G),U(1))$, and so serve as a generalization of the existing results on $\Aut(\D(G))$.

\section{Preliminaries and Notation}\label{sec:prelim}
Much of the paper will be dedicated to a large number of calculations and manipulations with identities, several of them of a fundamental nature.  Due to the large volume of such things, we will take pains to spell out a large number of definitions and results that might otherwise seem trivial or well-known.  Without a strong reminder and regular references to these, it can quickly become difficult to follow the subsequent calculations.

Throughout the paper $G,H$ will denote finite groups, unless otherwise specified.  We work over the field $\kk=\BBC$ of complex numbers.  Our reference for the general theory of Hopf algebras is \citep{Mon:HAAR}.  We use Sweedler notation for comultiplications whenever convenient. All unadorned tensor products are taken over $\kk$, and all undecorated $\Hom$'s, $\Aut$'s etc. are of quasi-Hopf algebras, Hopf algebras, or groups as appropriate.

\begin{df}
The vector space $\kk G$ and its linear dual $\du{G}$ are endowed with their usual Hopf algebra structures.
\end{df}

All units will be denoted by $1$, except for $\du{G}$ and $\widehat{G}$ whose units we denote by $\varepsilon$.   We also denote counits by $\varepsilon$, except for $\du{G}$ whose counit we denote by $\ev_1$.  All maps defined on vector spaces are linear maps, unless otherwise specified.
\begin{df}
  Given algebras $B,C$, a linear map $f\colon B\to C$ is said to be unital if $f(1)=1$.

  Dually, if $B,C$ are coalgebras, then a linear map $f\colon B\to C$ is said to be counital if $\varepsilon\circ f=\varepsilon$.

  If $B,C$ are both quasi-bialgebras, then a linear map $f\colon B\to C$ is said to be biunital if it is both unital and counital.
\end{df}
\begin{df}
  Suppose we are given vector spaces $V,W$, an algebra $A$, a coalgebra $C$, and linear maps $f\colon V\to A,g\colon W\to A$ and $h\colon C\to V, k\colon C\to W$.

  We say that $f$ and $g$ commute, written $f\comm g$, if $xy=yx$ for all $x\in\Img(f),y\in\Img(g)$.

  Dually, we say that $h,k$ cocommute, written $h\cocomm k$, if $h(c\com1)\ot k(c\com2) = h(c\com 2)\ot k(c\com 1)$ for all $c\in C$.
\end{df}

We denote a left action of the group $G$ on an object $X$ by $g\lact x$ and right actions by $x^g$ for $g\in G$ and $x\in X$.  For our purposes $X$ is usually either $G$, a normal subgroup of $G$, or $\du{G}$, all which are equipped with the standard left/right conjugation actions unless otherwise specified.  For example, for $x,g\in G$ $g^x = x\inv g x$ and $x\lact e_g = e_{xgx\inv}$.

We recall that the group-like elements of $\du{G}$ are the multiplicative $\kk$-linear characters of $G$, denoted $\widehat{G}$.  These are class functions, and so are pointwise fixed by the left/right conjugation actions of $G$ on $\du{G}$.  The following lemma, and the special case it singles out, will be critical for much of the paper, and is the fundamental reason why many of our results in \cref{sec:r-coalg,sec:r-alg} include assumptions that certain maps are coalgebra morphisms.

\begin{lem}\label{lem:coalg-to-group}
Let $G$ be a group and $K$ a coalgebra.  Any coalgebra morphism $r\colon\kk G\to K$ is equivalent to a set map $G\to G(K)$, where $G(K)$ denotes the group-like elements of $K$.  As a special case, if $K=\du{H}$ for a group $H$, then $r(g)\in\widehat{H}$ and $h\lact r(g) = r(g)$ for all $g\in G$, $h\in H$.
\end{lem}
\begin{proof}
  All coalgebra morphisms send group-like elements to group-like elements, and $\kk G$ is spanned by group-likes.  Thus $r$ is determined by the set map $G\to G(K)$ given by $g\mapsto r(g)$.  When $K=\du{H}$, we have $G(K)=\widehat{H}$.  The preceding remarks then give the final claim.
\end{proof}

\begin{cor}
  Any bialgebra map $\kk G\to \kk H$ is a morphism of Hopf algebras, and restricts to a group homomorphism $G\to H$.
\end{cor}
Indeed, all bialgebra morphisms between Hopf algebras are necessarily morphisms of Hopf algebras.  We will identity Hopf morphisms $\kk G\to \kk H$ with their associated group homomorphisms and vice versa.

Finally, we will make extensive use of the results in \citep{ABM,K14,KS14} as they concern $\Hom(\D(G),\D(H))$ and $\Aut(\D(G))$.  We will reference these results as needed, but will generally leave the full statements and proofs to the references.

\subsection{Quasi-bialgebras}\label{sub:qbs}
We now review the definitions of quasi-bialgebras and quasi-Hopf algebras, and the morphisms thereof.  These objects were introduced by \citet{Dri:QH}
\begin{df}
A quasi-bialgebra $B=(A,\Delta,\varepsilon,\phi)$ over a field $\kk$ is given by:
\begin{itemize}
  \item An associative $\kk$-algebra $A$ with unit $1$;
  \item Algebra morphisms $\Delta\colon A\to A\ot A$ and $\varepsilon\colon A\to\kk$ called the comultiplication and counit respectively;
  \item An invertible element $\phi\in A\ot A\ot A$, called the coassociator;
  \item Such that the following identities are satisfied for all $a\in A$:
  \begin{align*}
    (\id\ot\Delta)\circ\Delta(a) &= \phi\left[ (\Delta\ot\id)\circ\Delta(a)\right]\phi\inv,\\
    \left[(\id\ot\id\ot\Delta)(\phi)\right] \left[ (\Delta\ot\id\ot\id](\phi) \right] &= (1\ot \phi) \left[(\id\ot\Delta\ot\id)(\phi)\right] (\phi\ot 1),\\
    (\varepsilon\ot\id)\circ\Delta &= \id,\\
    (\id\ot\varepsilon)\circ\Delta&= \id,\\
    (\id\ot\varepsilon\ot\id)(\phi)&=1\ot 1.
  \end{align*}
\end{itemize}
\end{df}
We say that the comultiplication of a quasi-bialgebra is quasi-coassociative. The special case with $\phi=1$ makes the comultiplication coassociative and yields the usual definition of a bialgebra.  We call $\phi=1$ the trivial coassociator.  We identify a quasi-bialgebra with its underlying algebra (and vector space) whenever convenient.
\begin{rem}
  It is common in the literature to refer to $\phi$ as the associator.  This is because when we pass to the representation category $\Rep(B)$, the element $\phi$ describes the associator isomorphism $(X\ot Y)\ot Z \to X\ot (Y\ot Z)$ for representations $X,Y,Z$.  Since our emphasis is on the underlying quasi-bialgebra rather than the representations, and $\phi$ controls how close the comultiplication is to being coassociative, we opt to call $\phi$ the coassociator, instead.
\end{rem}

\begin{df}
  Let $K,L$ be quasi-bialgebras over $\kk$.  We say $f\colon K\to L$ is a morphism of quasi-bialgebras if the following all hold:
  \begin{enumerate}
    \item $f$ is an algebra morphism;
    \item $f$ preserves the coassociators: $f\ot f\ot f(\phi_K) = \phi_L$;
    \item $f$ respects the comultiplication: $f\ot f\circ\Delta_K = \Delta_L\circ f$;
    \item $f$ respects the counit: $\varepsilon_K = \varepsilon_L\circ f$.
  \end{enumerate}
  We say that $f$ is a morphism of coalgebras if it satisfies only the last two conditions.

  In either case $f$ is an isomorphism if it is also bijective.
\end{df}
\begin{df}
  We denote the set of quasi-bialgebra homomorphisms with $\Hom_{\text{QB}}$, and similarly we use $\Aut_{\text{QB}}$ for the group of quasi-bialgebra automorphisms.
\end{df}
Recall that, in contrast to standard Hopf algebras, in general $\Hom_{QB}(K,L)\neq \Hom(K,L)$ for quasi-Hopf algebras $K,L$.  In particular, we must directly verify a quasi-bialgebra morphism preserves antipodes (and $\alpha,\beta$ elements) to demonstrate that it is a morphism of quasi-Hopf algebras.
\begin{rem}
  One may be curious about relaxing the constraint that $f$ preserves the coassociators, by simply asking that it respect the identities it must satisfy in the obvious fashion.  The constraint is necessary to ensure that the induced functor of representation categories is strict, which ensures that the concept of isomorphisms is well-behaved with regards to the representation categories.
\end{rem}

This then leads us to the definition of a quasi-Hopf algebra, and morphisms thereof.

\begin{df}
  A quasi-Hopf algebra $H=(B,S,\alpha,\beta)$ over $\kk$ is a quasi-bialgebra $B=(A,\Delta,\varepsilon,\phi)$ over $\kk$ equipped with elements $\alpha,\beta\in A$ and a bijective anti-algebra morphism $S\colon A\to A$ (called the antipode).  Writing
  \begin{align*}
    \phi =& \sum \phi\ucom1\ot \phi\ucom2\ot \phi\ucom3,\\
    \phi\inv =& \sum \phi\ucom{-1}\ot \phi\ucom{-2}\ot \phi\ucom{-3},\\
    \Delta(a) =& \sum a\com1\ot a\com2,
  \end{align*}
  we require that all of the following identities be satisfied:
  \begin{align*}
   \varepsilon(a)\alpha&=\sum S(a\com1)\alpha a\com2;\\
   \varepsilon(a)\beta&=\sum a\com1 \beta S(a\com2);\\
   1&=\sum \phi\ucom1 \beta S(\phi\ucom2)\alpha \phi\ucom3;\\
   1&=\sum_j S(\phi\ucom{-1})\alpha \phi\ucom{-2} \beta S(\phi\ucom{-3}).
  \end{align*}
\end{df}
The special case $\alpha=\beta=1$ and $B$ a bialgebra yields the usual definition of a Hopf algebra.  It is possible to drop the bijectivity assumption from $S$, but in the cases we consider the bijectivity is guaranteed.  We identity a quasi-Hopf algebra with its underlying quasi-bialgebra, algebra, and/or vector space whenever convenient.

We then have the expected definition for a morphism of quasi-Hopf algebras.
\begin{df}
  If $K,L$ are quasi-Hopf algebras over $\kk$, we say $f\colon K\to L$ is a morphism of quasi-Hopf algebras if the following all hold:
  \begin{enumerate}
    \item $f$ is a morphism of quasi-bialgebras;
    \item $f$ respects the antipodes: $f\circ S_K= S_L\circ f$;
    \item $f(\alpha_K)=\alpha_L$;
    \item $f(\beta_K)=\beta_L$.
  \end{enumerate}
\end{df}
\begin{rem}
  In contrast to the case of bialgebra morphisms between Hopf algebras, a quasi-bialgebra morphism between quasi-Hopf algebras is not guaranteed to be a morphism of quasi-Hopf algebras.  So while we need only check the bialgebra morphism part for Hopf algebras, the quasi-Hopf case requires we check all parts of the definition.
\end{rem}
We can now proceed to provide the definitions for the quasi-Hopf algebra $\D^\omega(G)$.

Let $U(1)$ denote the multiplicative group of complex numbers of modulus 1.  A normalized 3-cocycle on $G$ is a function $\omega\colon G\times G\times G\to U(1)$ satisfying $\omega(g,1,h)=1$ as well as the cocycle condition
\begin{align}\label{eq:cocycle}
  \omega(a,b,c)\omega(a,bc,d)\omega(b,c,d) = \omega(ab,c,d)\omega(a,b,cd).
\end{align}
As a consequence, $\omega(1,g,h)=\omega(g,h,1)=1$.  We denote the set of all normalized 3-cocycles of $G$ by $Z^3(G,U(1))$.

Given a 3-cocycle $\omega\in Z^3(G,U(1))$ we also define
\begin{align}
    \theta_g(x,y) &= \frac{\omega(g,x,y)\omega(x,y,(xy)\inv g (xy))} {\omega(x, x\inv g x, y)},\label{eq:theta-def}\\
    \gamma_x(g,h) &= \frac{\omega(g,h,x) \omega(x,x\inv g x, x\inv h x)} {\omega(g, x, x\inv h x)}. \label{eq:gamma-def}
\end{align}
We call $\theta$ and $\gamma$ the multiplicative and comultiplicative phases, respectively.  Since $\omega$ is normalized, the phases must evaluate to 1 when any of their inputs is the identity.

The cocycle condition then implies the following identities hold:
\begin{gather}
    \theta_g(x,y)\theta_g(xy,z) = \theta_g(x,yz)\theta_{g^x}(y,z) \label{eq:assoc}\\
    \gamma_x(g,h)\gamma_x(gh,k) \omega(g^x, h^x, k^x) = \gamma_x(h,k)\gamma_x(g, hk) \omega(g,h,k) \label{eq:coassoc}\\
    \theta_g(x,y)\theta_h(x,y) \gamma_x(g,h) \gamma_y(g^x, h^x) = \theta_{gh}(x,y) \gamma_{xy}(g,h). \label{eq:struc-compat}
\end{gather}

Using these we may define a quasi-triangular quasi-Hopf algebra $\D^\omega(G)$ as in \citep{DPR}.

\begin{df}
Let $\omega\in Z^3(G,U(1))$.  As a vector space $\D^\omega(G)$ is $\du{G}\ot \kk G$.  We denote the standard basis elements of $\D^\omega(G)$ by $e_g\# x$, where $x\in G$ and $e_g\in\du{G}$ is the element dual to $g\in G$.  The multiplication of $\D^\omega(G)$ is given by
\[e_g\# x \cdot e_h\# y = \delta_{g,xhx\inv} \theta_g(x,y) e_g\# xy,\]
and has unit $\varepsilon\# 1$.  The comultiplication is given by
\[\Delta(e_g\# x) = \sum_{t\in G} \gamma_x(gt\inv,t) e_{gt\inv}\# x \ot e_t\# x,\]
with counit $\ev_1\ot\varepsilon$.  The antipode is given by
\[S(e_g\# x) = \theta_{g\inv}(x,x\inv)\inv \gamma_x(g,g\inv)\inv e_{x\inv g\inv x}\# x\inv,\]
with
\[ \alpha =1; \qquad \beta = \sum_{g\in G} \omega(g,g\inv,g) e_g\# 1.\]
Finally, the universal $R$-matrix is given by
\[R=\sum_{g\in G}e_g\# 1\ot \varepsilon\# g.\]
\end{df}
We note that we will not actually be interested in the universal $R$-matrix or the quasi-triangular property it provides here.  We have stated them for completeness only.

The identities for the phases guarantee that the multiplication is associative; that the comultiplication is quasi-coassociative with coassociator \[\phi = \sum_{a,b,c\in G} \omega(a,b,c)\inv e_a\#1\ot e_b\#1\ot e_c\#1;\] and that the comultiplication is an algebra morphism.  When $\omega\equiv 1$ is the trivial 3-cocycle, these identities reduce to the usual Hopf algebra structure of $\D(G)$.  In general, the properties of the category $\Rep(\D^\omega(G))$ depend only on the class of $\omega$ in $H^3(G,U(1))$, up to natural transformations.  However, the quasi-Hopf algebra structures themselves need not be isomorphic for cohomologous 3-cocycles.

The following standard lemma will also be used frequently, mainly in the case $z=1$ to show that the phases must vanish on certain combinations of inputs.
\begin{lem}\label{lem:average}
  Let $z_1,...,z_n,z\in \kk$ be complex numbers.  If $|z_i|\leq |z|$ for all $i$ and
  \[ \frac{1}{n}\sum_{i=1}^n z_i = z,\]
  then $z_i=z$ for all $i$.
\end{lem}
\begin{proof}
  This is a simple induction using the Cauchy-Schwarz inequality.
\end{proof}

Finally, when considering a map $\D^\omega(G)\to\D^\eta(H)$ we will adopt a priming convention for denoting phases and structural elements of $\D^\eta(H)$: the coassociator of $\D^\eta(H)$ is denoted $\phi'$, the phases are $\theta'$ and $\gamma'$, etc.  Unless it seems likely to introduce confusion, structure maps and morphisms will not be primed: the antipode remains $S$, the comultiplication remains $\Delta$, etc.  Whenever these need to be distinguished, we will add a subscript.  The main instance where we will need such subscripts is to distinguish between the antipode of $\D(G)$ from that of $\D^\omega(G)$ within an equation in \cref{sec:antipode}.

\section{Writing morphisms in components}\label{sec:components}
While the twisted double $\D^\omega(G)$ is fairly complex, it is built out of a pair of relatively nice pieces---$\du{G}$ and $\kk G$---modified by cohomological data.  It is natural, then, to wonder in what sense the morphisms between twisted doubles are built out of morphisms between these nice pieces, also modified by cohomological data.  This will decompose elements of $\genhom$ in a manner similar to the morphism decompositions given in \cite{ABM}.  Such decompositions have already been used to study $\Hom(\D(G),\D(H))$ and the group $\Aut(\D(G))$ \citep{K14,KS14,ABM,LenPri:Monoidal} to great effect.  While this approach will not be able to completely describe $\genhom$ in a succinct fashion---it will of necessity completely describe the morphisms in general, but in a "long list of inscrutable equations" sense---it will provide meaningful insights and constraints in general and complete information for isomorphisms when the 3-cocycles are sufficiently well-behaved.

We begin by defining $\du{G}_\omega$, for $\omega\in Z^3(G,U(1))$, to be the quasi-Hopf algebra with the same multiplication, unit, comultiplication, counit, and antipode as $\du{G}$, but endowed with the coassociator
\[ \sum_{a,b,c\in G} \omega(a,b,c)\inv e_a\ot e_b\ot e_c,\]
and $\alpha=1$, $\beta=\sum_{g\in G}\omega(g,g\inv,g)e_g$.

The first lemma is well-known and easily verified.

\begin{lem}\label{lem:obvious}
  \begin{lemenum}
    \item $\du{G}_\omega$ is (isomorphic to) a quasi-Hopf subalgebra of $\D^\omega(G)$ via the inclusion.
    \item $\ev_1\ot\id\colon \D^\omega(G)\to\kk G$ is a surjective morphism of quasi-Hopf algebras.
  \end{lemenum}
\end{lem}

The next lemma can be considered a specialized variant of \citep[Lemma 3.1]{ABM}, and is proved in much the same fashion.
\begin{lem}\label{lem:up-decomp}
    For any coalgebra morphism $\alpha\colon\du{G}_\omega\to \D^\eta(H)$ there exist linear maps $u\colon \du{G}\to\du{H}$ and $p\colon \du{G}\to\kk H$ such that $u$ is counital, $p$ is a coalgebra morphism, $u\cocomm p$, and
    \[ \alpha(e_g) = \sum_{t\in G}u(e_{gt\inv})\# p(e_t)\]
    for all $g\in G$.  Moreover, $u,p$ are unital whenever $\alpha$ is, and $p$ is a morphism of Hopf algebras whenever $\alpha$ is a morphism of quasi-bialgebras.
\end{lem}
\begin{proof}
  Let $\alpha\colon\du{G}_\omega\to\D^\eta(H)$ be a morphism of coalgebras, meaning $\alpha\ot\alpha\circ\Delta = \Delta\circ\alpha$ and $\ev_1\ot\varepsilon\circ\alpha=\ev_1$.

  We write $\alpha(e_g) = \sum_{h,x\in H} c_g(h,x) e_h\# x$ for scalars $c_g(h,x)\in\kk$.  We define \[u(e_g)=\id\ot\varepsilon\left(\alpha(e_g)\right) = \sum_{h,x\in H}c_g(h,x)e_h\] and \[p(e_g)=\ev_1\ot\id \left(\alpha(e_g)\right)=\sum_{x\in H}c_g(1,x)x.\]  We note that $\id\ot\varepsilon$ and $\ev_1\ot\id$ are both biunital.  So $u,p$ are automatically counital as they are compositions of counital maps, and they will be unital whenever $\alpha$ is unital.

  Now
  \[\Delta\alpha(e_g) = \sum_{h,k,x\in H} \gamma'_x(hk\inv,k)c_g(h,x)e_{hk\inv}\# x\ot e_k\# x\] is then equal to \begin{align*}
    \alpha\ot\alpha\Delta(e_g) &= \sum_{t\in G}\alpha\ot\alpha(e_{gt\inv}\ot e_t)\\
     &= \sum_{{\substack{h,h'\in H\\x,x'\in H\\t\in G}}} c_{gt\inv}(h,x)c_t(h',x')e_h\# x \ot e_{h'}\# x'.
  \end{align*}
  Applying $\id\ot\varepsilon\ot\ev_1\ot\id$ to both equations we obtain
  \begin{align*}
    \alpha(e_g) &= \sum_{h,x\in H }c_g(h,x) e_h\#x\\ &=
    \sum_{\substack{l\in G\\h,x,x'\in H}} c_{gl\inv}(h,x)c_l(1,x')e_h\# x'\\
    & = \sum_{l\in G}\Big(\sum_{h,x\in H} c_{gl\inv}(h,x)e_h\Big)\# \Big(\sum_{x'\in H}c_l(1,x')\varepsilon\# x'\Big)\\
    & = \sum_{l\in G} u(e_{gl\inv})\ot p(e_l).
  \end{align*}
  Applying $\ev_1\ot\id\ot\id\ot\varepsilon$ instead we get that $p\cocomm u$.  By \cref{lem:obvious} we further have that $p = \ev_1\ot\id \circ \alpha$ is a morphism of coalgebras, and indeed a morphism of quasi-bialgebras whenever $\alpha$ is a morphism of quasi-bialgebras.

  This completes the proof.
\end{proof}

\begin{thm}\label{thm:purv}
  Let $\psi\in\genhom$.  Then there exists a morphism of Hopf algebras $p\colon \du{G}\to\kk H$ and biunital linear maps $u\colon \du{G}\to\du{H}$, $r\colon \kk G\to\du{H}$, and $v\colon \kk G\to \kk H$ such that $p\cocomm u$ and satisfying
  \begin{align}\label{eq:morphdef}
    \psi(e_g\# x) &= \Big(\sum_{k\in A} u(e_{gk\inv})\# p(e_k)\Big)\cdot r(x)\# v(x)
  \end{align}
  for all $g,x\in G$.
\end{thm}
\begin{proof}
  Define biunital linear maps $\alpha\colon\du{G}_\omega\to \D^\eta(H)$ and $\beta\colon\kk G\to\D^\eta(H)$ by $\alpha(e_g)=\psi(e_g\# 1)$ and $\beta(x)=\psi(\varepsilon\# x)$.  Note that $\psi(e_g\# x) = \psi(e_g\# 1 \cdot \varepsilon\# x) = \psi(e_g\# 1)\psi(\varepsilon\# x)=\alpha(e_g)\beta(x)$.

  By \cref{lem:obvious} $\alpha$ is a morphism of quasi-Hopf algebras.  We can then take $u,p$ to be as in \cref{lem:up-decomp}.

  Since $\beta$ is a $\kk$-vector space morphism and $\kk G$ has the elements of $G$ as a basis, there exist linear maps $r_0\colon \kk G\to \du{H}$ and $v_0\colon \kk G\to \kk H$ such that $\beta(x) = r_0(x)\# v_0(x)$ for all $x\in G$.  Note that $\ev_1\ot\varepsilon\, \beta(x) = \ev_1(r_0(x))\varepsilon(v_0(x))=1$.  So by rescaling $r_0$ and $v_0$ as necessary, we may suppose we have a decomposition $\beta(x)=r(x)\# v(x)$ such that $r,v$ are both biunital, as desired.

  This completes the proof.
\end{proof}

Since the components $u,r,v$ are at this stage known only to be biunital linear maps, and the relationships between them are complex, we make the convention to expand each component in the standard basis elements whenever convenient as follows:
\begin{align}
  p(e_g) &= \sum_{b\in H}p(g,b)b,\label{eq:p-exp}\\
  u(e_g) &= \sum_{h\in H}u(g,h)e_h,\label{eq:u-exp}\\
  r(x) &= \sum_{h\in H}r(x,h)e_h,\label{eq:r-exp}\\
  v(x) &= \sum_{y\in H}v(x,y)y.\label{eq:v-exp}
\end{align}
The identities, written in terms of these coefficients, for $u,r,v$ to be biunital and for $p$ to be a morphism of Hopf algebras are recorded in \cref{sec:fundamental}.  They can also be found in \citet{ABM}.

Using these definitions, \cref{eq:morphdef} may be written
\begin{align}\label{eq:morphdef2}
    \psi(e_g\# x) = \sum u(e_{gk\inv})(b.r(x))\theta_j'(b,y)v(x,y)p(k,b) e_j\# by.
\end{align}

One may, at this stage, desire a nice list of necessary and sufficient conditions for a quadruple $\morphquad$ to yield an element $\psi\in\Hom(\D^\omega(G),\D^\eta(H))$, as defined by \cref{eq:morphdef}.  In the case that $\omega,\eta$ are both trivial such identities can be found in \citep{ABM,K14}.  We will find it necessary to explicitly write most of the identities that must hold in order to proceed with our analysis, in fact.  Most of these will be recorded in \cref{sec:coassoc,sec:bialg-ids}.  While a few of the identities are rather comprehensible, a nicer list than simply writing out the naive requirements seems ultimately inaccessible in the fully general case.

We will find fairly nice interpretations for $u,v$ in the following sections.  The main difficulty arises whenever the $r$ component is present.  The fundamental reason for this is that the canonical linear maps $\kk G\to \D^\omega(G)$ and $\D^\eta(H)\to \du{H}$ which are required in the definition of $r$ have no particular structure beyond being biunital.  They are neither algebra nor coalgebra morphisms in general.  At least one of the embeddings or projections involved in defining $u$ and $v$ have at least one of those two structures.  That $p$ was a morphism of Hopf algebras followed from both the relevant embedding and projection being well-behaved, in fact.  As such it is reasonable to expect that $u,v$ may admit fairly nice descriptions, whereas a nice description for $r$ can only be expected under a number of additional hypotheses.  For these reasons we tend to focus on the identities necessary for the results and examples this paper wishes to consider.

For the remainder of the paper we identify $\psi\in\genhom$ with its decomposition $\morphquad$.

\section{Fundamental identities for the components}\label{sec:fundamental}
Many of our results will require significant amounts of calculations to verify that various axioms are satisfied.  While many of the identities we provide in this section are either elementary or can be found in other references, they will be used with such frequency that we explicitly enumerate them.  The proofs will usually be either sketched, given a reference, or left as an easy verification for the reader.  We note that the component $p$ from \cref{thm:purv} will appear in virtually every calculation in the remainder of the paper.  We will therefore be using identities for $p$ with great frequency, and much of this section will be dedicated to properties and characterizations of $p$.  While we will explicitly reference the identities we use as often as possible, the reader may nevertheless find it convenient to be thoroughly familiar with them before proceeding.  Unless otherwise noted all references to $p,u,r,v$ refer to the decomposition $\morphquad$ for a fixed but arbitrary morphism $\psi\in\genhom$, though many of the results themselves will be stated in greater generality.

We first record the identities satisfied for the maps to be unital or counital.  These can be seen as special cases of identities given in \citep{ABM}.
\begin{lem}\label{lem:biunital}
  Let $p\colon \du{G}\to\kk H$, $u\colon \du{G}\to\du{H}$, $r\colon\kk G\to\du{H}$, and $v\colon \kk G\to\kk H$ be linear maps, with expansion in the standard bases given by \cref{eq:p-exp,eq:u-exp,eq:r-exp,eq:v-exp}.  Then the following all hold.
  \begin{lemenum}
    \item $p$ is unital if and only if \[\sum_{g\in G}p(g,h) = \delta_{1,h}\] for all $h\in H$.\label{eq:p-is-unital}
    \item $p$ is counital if and only if \[\sum_{h\in H}p(g,h) = \delta_{1,g}\] for all $g\in G$.\label{eq:p-is-counital}
    \item $u$ is unital if and only if \[\sum_{g\in G}u(g,h) = 1\] for all $h\in H$.\label{eq:u-is-unital}
    \item $u$ is counital if and only if \[u(g,1)=\delta_{1,g}\] for all $g\in G$.\label{eq:u-is-counital}
    \item $r$ is unital if and only if \[r(1,h)=1\] for all $h\in H$.\label{eq:r-is-unital}
    \item $r$ is counital if and only if \[r(g,1)=1\] for all $g\in G$.\label{eq:r-is-counital}
    \item $v$ is unital if and only if \[v(1,h)=\delta_{1,h}\] for all $h\in H$.\label{eq:v-is-unital}
    \item $v$ is counital if and only if \[\sum_{h\in H}v(g,h) = 1\] for all $g\in G$.\label{eq:v-is-counital}
  \end{lemenum}
  These results hold even when we replace $\du{G}$ and $\du{H}$ with $\du{G}_\omega$ and $\du{H}_\eta$ respectively.
\end{lem}
\begin{proof}
  The final claim holds because we do not change the identities or counits, only the coassociator.  The rest is a routine verification, so we establish only the first claim.  We have that $p$ is unital if and only if
  \begin{align*}
    p(\varepsilon) =& \sum_{g\in G} p(e_g)\\
    =& \sum_{g\in G}\sum_{h\in H}p(g,h)h \\
    =& 1.
  \end{align*}
  The claim then follows by comparing coefficients.
\end{proof}
\begin{rem}
  We will often apply the projections $\ev_1\ot\id$ and $\id\ot\varepsilon$ to entries in $\D^\eta(H)$, and/or the embeddings $\kk G\to \D^\omega(G)$, $\du{G}\to \D^\omega(G)$.  These are the basic reason why it is important that the components of $\morphquad$ are biunital, as otherwise these embeddings and projections would not be guaranteed to be well-behaved enough to be useful.
\end{rem}

By \cref{thm:purv} we know the component $p\colon\du{G}\to\kk H$ is a morphism of Hopf algebras, and such morphisms have a simple description.
\begin{lem}\label{lem:p-desc}\citep[Theorem 3.1]{K14}
Every Hopf algebra homomorphism $p\colon \du{G}\to\kk H$ is determined by abelian subgroups $A\subseteq G$ and $B\subseteq H$ and an isomorphism $\widehat{A}\to B$.  In particular, $p$ factors through an isomorphism $\du{A}\to\kk B$ via the canonical projection $\du{G}\to \du{A}$ and the canonical inclusion $\kk B\to \kk H$. As a consequence, $p(e_g)\neq 0 \iff g\in A$ and $\{p(e_a)\}_{a\in A}$ forms a basis for $\kk B$.
\end{lem}

The subgroups $A,B$ will occur frequently in the remainder of the paper.  Therefore any use of $A,B$ will always refer to the subgroups from the lemma unless otherwise specified.

At this point we wish to note that we will often find it convenient to investigate $u$ in terms of its linear dual $u^*\colon \kk H\to\kk G$, which is given by
\[u^*(h) = \sum_{g\in G} u(g,h)g\]
for any $h\in H$. The basic reason for this is that we find it easier to work with the group algebras, which are spanned by group-like elements.  Additionally, we will see that the components $u,v$ are related to each other, which makes it convenient to express things in terms of $u^*$ and $v$ rather than $u$ and $v$.  This is especially true whenever $u$ is an algebra morphism, equivalently $u^*$ is a coalgebra morphism. In this situation \cref{lem:coalg-to-group} applies to $u^*$, which forces many of the coefficients $u(g,h)$ in \cref{eq:u-exp} to be zero.

\begin{lem}\label{lem:u-simpler}
  Let $u\colon\du{G}\to\du{H}$ be an algebra morphism.  Then for all $h\in H$ there exists a unique $g\in G$ with $u(g,h)\neq 0$; moreover, in this case $u(g,h)=1$.  As a consequence, when $u$ is an algebra morphism $u(e_g)e_h\neq 0$ for some $g\in G$ and $h\in H$ if and only if $u^*(h)=g$.
\end{lem}
\begin{proof}
  By assumptions, $u^*\colon\kk H\to\kk G$ is a morphism of coalgebras.  By \cref{lem:coalg-to-group} this means that $u^*(h)\in G$ for all $h\in H$.  But we also have $u^*(h) = \sum_{g\in G} u(g,h)g$.  The claims now follow.
\end{proof}

For quasi-bialgebra morphisms $\morphquad\colon\D^\omega(G)\to\D^\eta(H)$, the subgroups $A,B$ determined by $p$ are strongly affected by the injectivity or surjectivity of the morphism.  This generalizations portions of \citep[Lemma 3.3]{K14}.
\begin{thm}\label{thm:inj-surj}
  Suppose $\psi=\morphquad\in\genhom$.  Then the following all hold.
  \begin{thmenum}
    \item If $\psi$ is surjective then $\kk B \Img(v) = \kk H$.
    \item If $\psi$ is injective then $\kk A \Img(u^*) = \kk G$ and $A\subseteq Z(G)$.
  \end{thmenum}
\end{thm}
\begin{proof}
  We recall \cref{eq:morphdef,eq:morphdef2}.

  For the first part, by assumption and \cref{lem:obvious} we have $\ev_1\ot\id\circ\psi\colon \D^\omega(G)\to \kk H$ is a surjective morphism of quasi-bialgebras satisfying $\ev_1\ot\id\circ\psi(e_g\# x) = p(e_g)v(x)$.  Thus the desired equality holds.

  For the second part, we have $\psi(e_g\#1) = \sum_{k\in G} u(e_{gk\inv})\# p(e_k)$.  In order for $\psi(e_g\# 1)\neq 0$ we must have that $\exists k\in A$, depending on $g$, such that $u(e_{gk\inv})\neq 0$.  This implies the desired equality.  Furthermore, since $u\cocomm p$, by dualizing we have that $u^*\comm p^*$. Thus $\Img(u^*)\subseteq C_{\kk G}(A)$.  But we have shown that $\kk G = \kk A\Img(u^*)\subseteq \kk A C_{\kk G}(A) = C_{\kk G}(A)$.  This implies $C_G(A)=G$, and so $A\subseteq Z(G)$ as desired.
\end{proof}

We now record all of the base identities that $p$ must satisfy in terms of its coefficients and the subgroups $A,B$.  Most of these can also be found in \citep{ABM}.

\begin{lem}\label{lem:p-ids-1}
  Let $p\colon\du{G}\to\kk H$ be a morphism of Hopf algebras, and expand $p$ in the standard basis elements as in \cref{eq:p-exp}.  Then the following all hold.
  \begin{lemenum}
    \item $p(g,h)=0$ if $g\not\in A$ or $h\not\in B$.
    \item $\sum_{a\in A}p(a,b) = \delta_{1,b}$ for all $b\in B$.\label{eq:p-unital}
    \item $\sum_{b\in B}p(a,b) = \delta_{1,a}$ for all $a\in A$.\label{eq:p-counital}
    \item $\sum_{t\in B} p(x,bt\inv)p(y,t) = \delta_{x,y} p(x,b)$ for all $x,y\in A$ and $b\in B$.\label{eq:p-alg}
    \item $\sum_{t\in A} p(xt\inv,b)p(t,c) = \delta_{b,c} p(x,b)$ for all $x\in A$ and $b,c\in B$.\label{eq:p-coalg}
    \item $p(a,b\inv)=p(a\inv,b)$ for all $a\in A$ and $b\in B$.\label{eq:p-antipode}
  \end{lemenum}
\end{lem}
\begin{proof}
  The first part follows from \cref{lem:p-desc}.  The next two parts are then equivalent, respectively, to the first two parts of \cref{lem:biunital}.  The fourth identity is equivalent to $p(e_x\cdot e_y) = \delta_{x,y}p(e_x)$, which is equivalent to $p$ being an algebra morphism.  The fifth part similarly follows from $p$ being a coalgebra morphism.  The last part follows from the preservation of the antipodes:
  \[ \sum_{b\in B} p(a\inv,b)b = p(e_{a\inv}) = p(S(e_a)) = Sp(e_a) = \sum_{b\in B} p(a,b)b\inv = \sum_{b\in B} p(a,b\inv)b.\]
\end{proof}
We can improve on the identities for the algebra and coalgebra morphism conditions, though, which will be necessary in a few manipulations later on.  First, we wish to investigate what $p$ looks like when expressed in a different basis for $\du{A}$.

\begin{df}\label{df:chi-def}
  Let $p\colon\du{G}\to\kk H$ be a morphism of Hopf algebras with associated abelian groups $A,B$.  For $b\in B$ we define $\chi_b\in\du{A}$ to be the unique element of $\du{A}$ such that $p(\chi_b)=b$.  We necessarily have $\chi_b\in\widehat{A}$, the linear characters of $A$.
\end{df}
This is equivalent to restricting the domain and codomain of $p$ to the isomorphism $\du{A}\to\kk B$ and setting $\chi_b=p\inv(b)$.  By bijectivity of this restriction, $\{\chi_b\}_{b\in B}=\widehat{A}$ is a basis of $\du{A}$.  Indeed, since $p$ is a morphism of Hopf algebras we have a group structure $\chi_b\cdot \chi_c = \chi_{bc}$ which agrees with the usual multiplication in $\widehat{A}$, where $\varepsilon=\chi_1$ is the identity.

The following shows how we can convert from the standard basis of $\du{A}$ to the character basis, as indexed by $p$.
\begin{lem}\label{lem:basis-std-in-chi}
  Suppose we are given a morphism of Hopf algebras $p\colon \du{G}\to\kk H$, with the abelian groups $A,B$ as before.  Then for all $a\in A$
  \[ e_a = \sum_{b\in B} p(a,b)\chi_b.\]
\end{lem}
\begin{proof}
  Applying $p$ to the right-hand side we have
  \[ \sum_{b\in B} p(a,b)b,\]
  which is precisely the definition of $p(e_a)$.  Since $p$ is bijective when restricted to $\du{A}\to\kk B$, this proves the claim.
\end{proof}
We can, of course, convert from the character bases back to the standard basis, as well.  This process yields useful information about the coefficients of $p$.
\begin{lem}\label{lem:basis-chi-in-std}\label{lem:p-norms}
    Let $p\colon\du{G}\to\kk H$ be a morphism of Hopf algebras.  Then for all $a\in A$ and $b\in B$ the following hold:
    \begin{enumerate}
      \item $p(1,b)=p(a,1)=1/|A|$;
      \item $|p(a,b)|=1/|A|$;
      \item $\displaystyle{\chi_b = |A| \sum_{c\in A} p(c\inv,b)e_c}$.
    \end{enumerate}
\end{lem}
\begin{proof}
  Let $f(a,b)\in\kk$ be scalars for each $a\in A$ and $b\in B$ such that
  \[ \chi_b = \sum_{a\in A} f(a,b) e_a.\]
  Applying the previous lemma, we have
  \begin{align*}
    \chi_b &= \sum_{a\in A} f(a,b) e_a\\
    &= \sum_{a\in A}\sum_{c\in B} f(a,b) p(a,c)\chi_c.
  \end{align*}
  Since the $\chi_b$ form a basis, this means
  \[ \sum_{a\in A} f(a,b)p(a,c) =\delta_{c,b}\]
  for all $b,c\in B$.  Now by \cref{lem:p-ids-1} we see that this is (uniquely) solved by
  \[ f(a,b) = \frac{p(a\inv,b)}{p(1,b)}.\]
  Since the $\chi_b$ are linear characters of $A$, the $f(a,b)$ must all have norm 1.  So we conclude that
  \[ |p(a,b)|=|p(1,b)|\]
  for all $a\in A$ and $b\in B$.  Considering the special case $b=1$, by \cref{eq:p-unital} and \cref{lem:average} we conclude that
  \[ p(a,1) = \frac{1}{|A|}\]
  for all $a\in A$.  Applying the same argument to the dual $p^*\colon \du{H}\to \kk G$, which is also a morphism of Hopf algebras, yields
  \[ p(1,b) = \frac{1}{|A|}\]
  for all $b\in B$.  This gives the desired formula for $\chi_b$ and also proves that
  \[ |p(a,b)| = \frac{1}{|A|}\]
  for all $a\in A$ and $b\in B$.

  This completes the proof.
\end{proof}

With the preceding lemma we can obtain the desired improvements over \cref{eq:p-alg,eq:p-coalg}.

\begin{thm}\label{thm:p-simpler}
  Let $p\colon \du{G}\to\kk H$ be a morphism of Hopf algebras.  Then
  \begin{align*}
    p(a,bc) =& |A| p(a,b)p(a,c)\\
    p(ax,b) =& |A| p(a,b)p(x,b)
  \end{align*}
  for all $a,x\in A$ and $b,c\in B$.
\end{thm}
\begin{proof}
  Let $b,c\in B$.  Then we have
  \begin{align*}
    S(\chi_b)S(\chi_c) &= |A|^2 \sum_a p(a,b)p(a,c) e_a\\
    &= S(\chi_c\chi_b)\\
    &= S(\chi_{cb})\\
    &= |A|\sum_a p(a,cb)e_a.
  \end{align*}
  Since $B$ is abelian, comparing coefficients yields the first identity.  The second is obtained by replacing $p$ with $p^*$.
\end{proof}
\begin{cor}
  For all $a\in A$ and $b\in B$, $p(a,b\inv) =\overline{p(a,b)}$, where $\overline{z}$ denotes the complex conjugate of $z$.
\end{cor}
\begin{proof}
  Apply the preceding theorem and \cref{lem:p-norms} to $p(a,b)p(a,b\inv)$.
\end{proof}
\begin{df}
  Let $\kk^\times$ be the multiplicative group of units in $\kk$.  Given groups $G,H$, a map $f\colon G\times H\to \kk^\times$ is said to be a bicharacter if the following two conditions hold:
  \begin{itemize}
    \item For all $g\in G$ $f(g,\cdot)\in \widehat{H}$;
    \item For all $h\in H$, $f(\cdot,h)\in\widehat{G}$.
  \end{itemize}
  Equivalently, for all $x,y\in G$ and $a,b\in H$ the following both hold:
  \begin{itemize}
    \item $f(xy,a) = f(x,a)f(y,a)$.
    \item $f(x,ab) = f(x,a)f(x,b)$.
  \end{itemize}
\end{df}

\begin{cor}
  Suppose $p\colon\du{G}\to\kk H$ is a morphism of Hopf algebras.  Then the map $\sigma\colon A\times B\to\kk^\times$ given by
  \[ \sigma(a,b)\mapsto |A| p(a,b)\]
  is a bicharacter.
\end{cor}
\begin{proof}
  The preceding theorem implies that $\sigma(aa',b) =\sigma(a,b)\sigma(a',b)$ and $\sigma(a,bb')=\sigma(a,b)\sigma(a,b')$ for all $a,a'\in A$ and $b,b'\in B$.  This is precisely the definition that $\sigma$ is a bicharacter.
\end{proof}
\begin{df}
  Let $G,H$ be abelian groups. A bicharacter $\sigma\colon G\times H\to\kk^\times$ is said to be orthogonal if
  \[ \langle\,\sigma(x,\cdot),\sigma(y,\cdot)\,\rangle = \delta_{x,y}\]
  and
  \[ \langle\, \sigma(\cdot,k),\sigma(\cdot,l)\,\rangle = \delta_{k,l}\]
  for all $x,y\in G$ and $k,l\in H$, where the inner products are the usual inner products of characters over $\kk$.
\end{df}
As defined, the assumption that $G,H$ are abelian is required.  For example, if $1\neq x\in G'$ then necessarily $\sigma(x,\cdot)=\varepsilon$ is the trivial character, which therefore violates the orthogonality requirements.  While we will not need it here, a non-trivial definition for arbitrary $G,H$ is acquired by simply requiring that the canonical map $G/G'\times H/H'\to \kk$ be an orthogonal bicharacter, as defined above.
\begin{thm}
  The Hopf algebra morphisms $p\colon\du{G}\to \kk H$ with given $A,B$ are in bijective correspondence with the orthogonal $A\times B$ bicharacters.
\end{thm}
\begin{proof}
  The previous corollary shows how to convert between $p$ and a bicharacter.  Explicitly, the relation is given by
  \[ \sigma(a,b)=|A|p(a,b),\]
  for $a\in A$ and $b\in b$, and $p(a,b)=0$ otherwise.  We need to show that given $p$, the obtained $\sigma$ is orthogonal.  And that given $\sigma$, the $p$ obtained defines a morphism of Hopf algebras only if $\sigma$ is orthogonal.  By the preceding corollaries and theorem, $\sigma$ is an $A\times B$ bicharacter either by assumption or by properties of $p$.  We claim that \cref{eq:p-alg,eq:p-coalg} are equivalent to the desired orthogonality condition on $\sigma$.  We consider only \cref{eq:p-alg} here, as the other case is nearly identical.  The reader may find it worthwhile to also express \cref{eq:p-unital,eq:p-counital} in these terms, but we do not need to explicitly consider them because they follow from \cref{eq:p-alg,eq:p-coalg}.  So with $\sigma$ and $p$ related as before, so that $\sigma$ is always a bicharacter, we have
  \begin{align*}
    \sum_{t\in B} p(x,bt\inv) p(y,t) &= \frac{1}{|A|^2} \sum_{t\in B}\sigma(x,bt\inv)\sigma(y,t)\\
    &= \frac{1}{|A|^2} \sum_{t\in B}\sigma(x,b)\sigma(x,t\inv)\sigma(y,t)\\
    &= \frac{\sigma(x,b)}{|A|^2} \sum_{t\in B}\sigma(x,t\inv)\sigma(y,t)\\
    &= \frac{p(x,b)}{|A|} \sum_{t\in B} \overline{\sigma(x,t)}\sigma(y,t)\\
    &= p(x,b)\langle\, \sigma(y,\cdot),\sigma(x,\cdot)\,\rangle.
  \end{align*}
  Therefore by \cref{eq:p-alg} and the definition of orthogonal bicharacter, this last expression is equal to $\delta_{x,y}p(x,b)$ for all $x,y\in A$ and $b\in B$ if and only if $\sigma$ is orthogonal, if and only if $p$ is a morphism of Hopf algebras.
\end{proof}

\section{Quasi-Hopf algebras determined by \ensuremath{p}}\label{sec:p-quotients}\label{sec:p-embeddings}
While the previous section offers some rather nice descriptions for the component $p$, we will find it useful for understanding the components $u,v$ to have a couple of other perspectives.

So suppose we have a morphism of Hopf algebras $p\colon\du{G}\to\kk H$ and 3-cocycles $\omega\in Z^3(G,U(1))$, $\eta\in Z^3(H,U(1))$.  We define $\chi_b\in \widehat{A}$ as in \cref{df:chi-def}.

\begin{df}\label{df:p-quotients}
  We define a quasi-Hopf algebra $\du{H}\#_p^\eta\du{A}$ in the following fashion.  The underlying vector space is $\du{H}\ot \du{A}$, with basis elements $e_h\# \chi_b$ for $h\in H$ and $b\in B$.  The multiplication of $\du{H}\#_p^\eta\du{A}$ is given by
  \[ e_x\# \chi_b\cdot e_y\# \chi_{b'} = \delta_{x,byb\inv} \theta_x'(b,b')e_x\# \chi_{bb'},\]
  with unit $\varepsilon\# \varepsilon=\varepsilon\#\chi_1$. The comultiplication of $\du{H}\#_p^\eta\du{A}$ is defined by
  \[ \Delta e_h\# \chi_b = \sum_{t\in H} \gamma_b'(ht\inv,t) e_{ht\inv}\# \chi_b\ot e_t\#\chi_b,\]
  with counit $\ev_1\ot \ev_1$.  The coassociator is
  \[ \sum_{x,y,z\in H} \eta(x,y,z)\inv e_x\#\chi_1\ot e_y\#\chi_1\ot e_z\#\chi_1.\]
  The antipode is
  \[ S(e_h\#\chi_b) = \theta_{h\inv}(b,b\inv)\inv \gamma_b(h,h\inv)\inv e_{b\inv h\inv b}\# \chi_{b\inv},\]
  with
  \[ \alpha=\varepsilon\# \chi_1,\qquad \beta=\sum_{h\in H} \eta(h,h\inv,h)e_h\# \chi_1.\]
\end{df}
Since $p$ is a morphism of Hopf algebras, this defines a quasi-Hopf algebra for the same reasons that the structures of $\D^\eta(H)$ define a quasi-Hopf algebra.  Indeed, the following can be used for an alternative definition of $\du{H}\#_p^\eta\du{A}$ by using the standard pullback construction.
\begin{cor}\label{cor:p-quot-morph}
  The linear map $\id\ot p\colon \du{H}\#_p^\eta\du{A}\to \D^\eta(H)$ is an injective morphism of quasi-Hopf algebras.  The image is the quasi-Hopf subalgebra $\du{H}\#\kk B$ of $\D^\eta(H)$.
\end{cor}
\begin{proof}
  This is immediate from the definitions.
\end{proof}

\begin{df}\label{df:p-embeddings}
  Suppose $A$ is a normal subgroup of $G$.  We define a quasi-Hopf algebra $\kk B\#_p^\omega\kk G$ in the following fashion.  The underlying vector space is $\kk B\ot \kk G$, with basis elements $p(e_a)\# g$ for $a\in A$, $g\in G$.  The multiplication of
  $\kk B\#_p^\omega \kk G$ is defined by
  \[ p(e_a)\# x \cdot p(e_b)\# y = \delta_{a,xbx\inv} \theta_a(x,y) p(e_a)\# xy,\]
  with unit $1\#1 = \sum_{a\in A} p(e_a)\# 1$.  The comultiplication is given by
  \[ \Delta(p(e_a)\# x) = \sum_{c\in A} \gamma_x(ac\inv,c) p(e_{ac\inv})\# x \ot p(e_c)\# x,\]
  with counit $\varepsilon\ot\varepsilon$.  The coassociator is
  \[ \sum_{a,b,c\in A} \omega(a,b,c)\inv p(e_a)\#1\ot p(e_b)\#1 \ot p(e_c)\#1.\]
  The antipode is
  \[ S(p(e_a)\# x) = \theta_{a\inv}(x,x\inv)\inv \gamma_x(a,a\inv)\inv p(e_{x\inv a\inv x})\# x\inv,\]
  with
  \[ \alpha=1,\qquad \beta=\sum_{a\in A}\omega(a,a\inv,a)p(e_a)\# 1.\]
\end{df}
For the necessity of the normality of $A$, note that $1\# 1$ can be a zero divisor when $A$ is not normal.  For if $A$ is not normal then $xbx\inv\not \in A$ is possible for some $b\in A$ and $x\in G$.  In \cref{lem:vp-comm} we will show that $A$ will always be normal for our purposes.

As before, since $p$ is a morphism of Hopf algebras, this is a well-defined quasi-Hopf algebra for the same reasons that the structures of $\D^\omega(G)$ define a quasi-Hopf algebra.  The explicit definition above is tailor made for the following result to hold.
\begin{cor}\label{cor:p-embed-morph}
  When $A$ is normal in $G$, the linear map \[p\ot\id\colon \D^\omega(G)\to \kk B\#_p^\omega\kk G\] is a surjective morphism of quasi-Hopf algebras.  It is an isomorphism when the domain is restricted to the quasi-Hopf subalgebra $\du{A}\# \kk G$ of $\D^\omega(G)$.
\end{cor}
\begin{proof}
  This is immediate from the definitions.
\end{proof}
\begin{rem}
We have presented the structures explicitly, rather than rely on invoking the standard pullback and push forward constructions alone, as we will tie the properties of these structures into the properties of the components $v,u$ in \cref{sec:uv-bialg}.  One can change the bases for these quasi-Hopf algebras by using \cref{lem:basis-chi-in-std,lem:basis-std-in-chi,eq:p-exp}.  We have chosen here the bases which makes the connection to the twisted doubles the most transparent.
\end{rem}

\begin{lem}\label{lem:p-embed-coassoc}
  Suppose $A$ is normal in $G$.  Then $\kk B\#_p^\omega\kk H$ is a Hopf algebra if and only if $\omega$ restricts to the trivial 3-cocycle on $A$.
\end{lem}
\begin{proof}
  It suffices to show that the coassociator and $\beta$ elements are both trivial.  The coassociator of $\kk B\#_p^\omega\kk H$ is \[ \sum_{a,b,c\in A} \omega(a,b,c)\inv p(e_a)\# 1\ot p(e_b)\# 1\ot p(e_c)\# 1,\]
  which when expanded in the standard basis is\[ \sum_{\substack{a,b,c\in A\\x,y,z\in B}} \omega(a,b,c)\inv p(a,x)p(b,y)p(c,z)\, x\# 1 \ot y\# 1\ot z\#1.\]
  If this equals the trivial coassociator, then considering the $x=y=z=1$ term of the coassociator, we see
  \[ \frac{1}{|A|^3} \sum_{a,b,c\in A} \omega(a,b,c)\inv = 1.\]
  By \cref{lem:average} this is equivalent to $\omega(a,b,c)=1$ for all $a,b,c\in A$.  For the reverse direction, given that $\omega\equiv 1$ on $A^3$, \cref{eq:p-unital} shows that the coassociator is trivial.  Finally, note that $\beta=1$ whenever $\omega\equiv 1$ on $A^3$, as desired.  This completes the proof.
\end{proof}
In \cref{lem:gen-omega-A}, we will see that $\kk B\#_p^\omega\kk H$ is a Hopf algebra whenever $p$ is such that $\exists$ $u,r,v$ with $\morphquad\in\genhom$.  On the other hand, the coassociator of $\du{H}\#_p^\eta \du{A}$ is trivial if and only if $\eta$ is trivial.

\begin{rem}\label{rem:alt-alg}
Here's an alternative, and equivalent, approach for producing the algebra structure of $\kk B\#_p^\omega \kk G$.

Suppose $A$ is normal in $G$.  Let $(\kk B)^\times$ denote the abelian, multiplicative group of units of the group algebra $\kk B$.

Considering \cref{eq:pv-alg} we define a function $\beta\colon G\times G\to (\kk B)^\times$ by
\begin{align}\label{eq:beta}
    \beta(x,y) = p(\sum_{g\in G} \theta_g(x,y) e_g) = p(\sum_{g\in A} \theta_g(x,y) e_g).
\end{align}
We note that \[p(\sum_{g\in A} \theta_g(x,y) e_g)\inv = p(\sum_{g\in A} \theta_g(x,y)\inv e_g),\]
and so the function is well-defined.  $G$ acts on $\du{A}$ via (left) conjugation since $A$ is normal.  It is well-known, and easily checked, that the action of $G$ on $\du{A}$ sends  $\widehat{A}$ to itself.  Since this is a basis of $\du{A}$ which is mapped (isomorphically) to $B$ by $p$, we see that the action of $G$ on $\du{A}$ defines an action of $G$ on $B$.  In other words, we may define $g\lact b$ to be $p(g\lact \chi_b)$ for all $g\in G$ and $b\in B$.  This gives an action of $G$ on $\kk B$ (equivalently defined by $g\lact p(e_a)=p(e_{xax\inv})$).  Since $(\kk B)^\times$ is characterized as those elements of $\kk B$ all of whose coefficients in the basis $\{p(e_a)\}_{a\in A}$ are non-zero, this action in turn restricts to an action on $(\kk B)^\times$.  Furthermore, \cref{eq:assoc} implies that $\beta$ defines a 2-cocycle of $G$ with values in $(\kk B)^\times$.  We can therefore define an associative multiplication (with identity) on $\kk B\ot \kk G$ by
\[ b\ot x \cdot c\ot y = b[x\lact c] \beta(x,y)\ot xy\]
for all $b,c\in B$ and $x,y\in G$.  A straightforward application of \cref{lem:basis-chi-in-std,lem:basis-std-in-chi} shows that this is precisely the algebra structure of $\kk B\#_p^\omega \kk G$.
\end{rem}

\begin{rem}\label{rem:alt-coalg}
Similarly, the following is an alternative approach to obtaining the coalgebra structure of $\kk B\#_p^\omega \kk G$.

For each $x\in G$ define $T(x)\in \kk B\ot\kk B$ by
\begin{align}\label{eq:T-def}
    T(x) = \sum_{g,t\in A} \gamma_x(gt\inv,t)p(e_{gt\inv})\ot p(e_t).
\end{align}
This element is invertible with inverse
\[ T(x)\inv = \sum_{g,t\in A} \gamma_x(gt\inv,t)\inv p(e_{gt\inv})\ot p(e_t).\]
We write $T(x) = \sum T(x)\ucom1\ot T(x)\ucom2.$

The comultiplication $\Delta$ on $\kk B\#_p^\omega \kk G$ is then given by
\[ \Delta(b\# x) = (b T(x)\ucom 1\# x) \ot (b T(x)\ucom2\# x),\]
which can again be verified by converting between bases for $\kk B$. Indeed, the quasi-coassociativity of this comultiplication is given by a 2-cycle like condition on $T$, which follows from \cref{eq:coassoc}.  It is precisely a 2-cycle condition, yielding a coassociative comultiplication, whenever the coassociator of $\kk B\#_p^\omega \kk G$ is trivial.
\end{rem}

\section{Coassociator relations}\label{sec:coassoc}
We recall that any quasi-bialgebra morphism must map the coassociator of the domain to the coassociator of the codomain.  In the case of a morphism between twisted group doubles, this condition imposes various restrictions on the 3-cocycles and the subgroup $A$, which we will now explore.

Let $\omega\in Z^3(G,U(1))$.  Consider any linear map $f\colon \kk H\to\kk G$.  We may linearly extend the 3-cocycle $\omega$ to a function $\kk G^3\to \kk$.  This extension of $\omega$ allows us to define $\omega^f = \omega\circ f\ot f\ot f$, which is a linear map $\kk H^3\to \kk$.

For any $\morphquad\in\genhom$ the coassociator condition means
\begin{align}\label{eq:morph-coassoc}
    \sum_{\substack{a,b,c\in A\\ x,y,z\in G}}&\omega(ax,by,cz)\inv u(e_x)\# p(e_a) \ot u(e_y)\# p(e_b)\ot u(e_z)\# p(e_c)\\
        &= \sum_{h,m,n\in H} \eta(h,m,n)\inv e_h\# 1\ot e_m\#1 \ot e_n\# 1.\nonumber
\end{align}

Our first lemma shows how $\morphquad$ relates the 3-cocycles to each other.
\begin{lem}\label{lem:gen-u-omega-eta}
  Let $\morphquad\in\genhom$.  Then $(\omega\inv)^{u^*}=\eta\inv$.  When $u$ is a morphism of algebras this is equivalent to \[\eta(x,y,z)=\omega(u^*(x),u^*(y),u^*(z))\] for all $x,y,z\in H$.
\end{lem}
\begin{proof}
  Apply $(\id\ot\varepsilon)^3$ to \cref{eq:morph-coassoc}, expanded in the standard basis elements.  The obtained equality is easily shown to be equivalent to $(\omega\inv)^{u^*}=\eta\inv$.  The final claim follows from \cref{lem:u-simpler}.
\end{proof}

The next lemma gives a general connection between the values of $\omega$ and the component $p$.
\begin{lem}\label{lem:gen-omega-A}
  Let $\morphquad\in\genhom$.  Then $\omega$ restricts to the trivial 3-cocycle on $A$.
\end{lem}
\begin{proof}
  Applying $(\ev_1\ot\id)^3$ to \cref{eq:morph-coassoc}, expanded in standard basis elements, we get
  \[\sum_{\substack{a,b,c\in A\\ a',b',c'\in B}} \omega(a,b,c)\inv p(a,a')p(b,b')p(c,c')a'\ot b'\ot c' = 1\ot 1\ot 1.\]
  Considering the $a'=b'=c'=1$ term in the summation we get \[\frac{1}{|A|^3}\sum_{a,b,c\in A}\omega(a,b,c)\inv = 1.\]  From \cref{lem:average} we conclude that $\omega\equiv 1$ on $A^3$, as desired. Indeed, by \cref{eq:p-unital,eq:p-counital} this condition guarantees that the remaining terms in the sum vanish, as necessary.
\end{proof}

When the component $u$ has nice structure, then preservation of the coassociators puts fairly strong restrictions on the 3-cocycles, and conversely.
\begin{thm}\label{thm:gen-omega-u}
Let $\morphquad\in\genhom$.
    \begin{enumerate}
        \item If $u$ is a morphism of algebras, then $\omega(x,y,z)=1$ whenever $\{x,y,z\}\in A\cup \Img(u^*)$ and $\{x,y,z\}\cap A\neq\emptyset$.
        \item If $u$ is a morphism of Hopf algebras then $\omega(ax,by,cz) = \omega(x,y,z)$ for all $a,b,c\in A$ and $x,y,z\in\Img(u^*)$.
    \end{enumerate}
    In particular, if $\morphquad$ has $u$ a morphism of Hopf algebras then by restriction $\omega$ defines an element of $Z^3((A\Img(u^*))/A,U(1))$.  As a special case, when $A\Img(u^*)=G$ then we have $\omega \in Z^3(G/A,U(1))$.
\end{thm}
\begin{proof}
  Applying $\ev_1\ot\id\ot(\id\ot\varepsilon)^{\ot 2}$ to \cref{eq:morph-coassoc} we get

  \[ \sum \omega(a,u^*(j),u^*(k))\inv p(a,b)b\ot e_j\ot e_k = 1\ot\varepsilon\ot\varepsilon.\]
  Considering the term $b=1$ on the left hand side we obtain for all $j,k$ that
  \[ \frac{1}{|A|}\sum_{a\in A} \omega(a,u^*(j),u^*(k))\inv  = 1,\]
  so by \cref{lem:average} $\omega(a,x,y)=1$ for all $a\in A$ and $x,y\in \Img(u^*)$.  Similar arguments after applying the various combination of projections show that $\omega(x,y,z)=1$ whenever $\{x,y,z\}\in A\cup \Img(u^*)$ and $\{x,y,z\}\cap A\neq\emptyset$, which is the first part.

  The second part then follows from the first by repeated application of the 3-cocycle condition.

  Now we can observe that $p\cocomm u$ implies that $A$ is in the center of $\kk A\Img(u^*)$.  Thus $(A\Img(u^*))/A$ is a well-defined quotient group when $u$ is a morphism of Hopf algebras.  The remaining claims are then immediate consequences of the second part.
\end{proof}

\section{Basic Identities for \ensuremath{\morphquad}}\label{sec:bialg-ids}
Let $\psi=\morphquad\in\genhom$.  In this section we record the remaining identities the components must satisfy in order to define a morphism of quasi-bialgebras.  We will not prove any results here, even though a few are rather obvious.  As such, the reader may safely skip to the next section and simply refer back to the appropriate equations as they are referenced.  The antipode relations for morphisms of quasi-Hopf algebras will be considered in \cref{sec:antipode}.  We will make liberal use of \cref{lem:biunital,eq:morphdef2,eq:p-exp,eq:u-exp,eq:r-exp,eq:v-exp} throughout this section.

To begin, we consider the identities for $\morphquad$ to be an algebra morphism.  For $g,g',x,x'\in G$ we have the equality of
\begin{align}\label{eq:alg-inner}
  \psi(e_g\# x \cdot e_{g'}\# x') = \delta_{g,xg'x\inv} \theta_g(x,x') \sum_{\substack{t\in A,\\ b\in B,\\j,y\in H}}  &u(e_{gt\inv})[b\lact r(xx')]p(t,b)\\
  & v(xx',y)\theta'_j(b,y) e_j\# by \nonumber
\end{align}
and
\begin{align}\label{eq:alg-outer}
  \psi(e_g\# x)\psi(e_{g'}\# x') = \sum_{\substack{t,t'\in A\\b,b'\in B\\j,y,y'\in H}} \ &u(e_{gt\inv})\left[(by)\lact u(e_{g't^{\prime -1}})\right] p(t,b)p(t',b')\\
   &[b\lact r(x)]\left[(byb')\lact r(x')\right]v(x,y)v(x',y') \nonumber\\
   &\theta'_j(b,y)\theta'_{j^{by}}(b',y') \theta'_j(by,b'y') e_j\# byb'y'.\nonumber
\end{align}

Now we apply $\ev_1\ot\id$ to \cref{eq:alg-inner,eq:alg-outer}.  This yields
\begin{align}\label{eq:pv-alg}
\delta_{g,xg'x\inv} \theta_g(x,x') p(e_g)v(xx') &= p(e_g)v(x) p(e_{g'}) v(x').
\end{align}

Next we apply $\id\ot\varepsilon$ to \cref{eq:alg-inner,eq:alg-outer} to get that the following two expressions are equal:
\begin{align}\label{eq:ur-mult-1}
  \delta_{g,xg'x\inv}\theta_g(x,x') \sum_{\substack{t\in A\\ b\in B\\j,y\in H}} u(e_{gt\inv})[b\lact r(xx')]p(t,b)v(xx',y)\theta'_j(b,y)e_j,
\end{align}
\begin{align}\label{eq:ur-mult-2}
  \sum_{\mathclap{\substack{t,t'\in A\\ b,b'\in B\\j,y,y'\in H}}} & u(e_{gt\inv})\left[(by)\lact u(e_{g't^{\prime -1}})\right][b\lact r(x)][(byb')\lact r(x')]p(t,b)p(t',b')\\
  & \ v(x,y)v(x',y') \theta'_j(b,y)\theta'_{j^{by}}(b',y') \theta'_j(by,b'y') e_j. \nonumber
\end{align}

We now consider the identities that must hold to determine a morphism of coalgebras.  We must have the equality of
\begin{align}\label{eq:coalg-outer}
  \Delta\psi(e_g\# x) = \sum_{\substack{t\in A\\x',j,l\in H\\b\in B}} & \theta'_j(b,x')\gamma'_{bx'}(jl\inv,l)p(t,b)v(x,x')u(gt\inv,j)r(x,b\inv j b)\\
   & \ e_{jl\inv}\# bx'\ot e_l\# bx', \nonumber
\end{align}
and
\begin{align}
  \psi\ot\psi\Delta(e_g\# x) &= \sum_{t\in G}\gamma_x(gt\inv,t)\cdot(\psi\ot\psi)(e_{gt\inv}\# x\ot e_t\# x)\nonumber\\
  &= \sum_{\substack{t\in G\\l,k\in A\\x',y',j,j'\in H\\b,b'\in B}} \gamma_x(gt\inv,t)\theta'_j(b,x')\theta'_{j'}(b',y')\label{eq:coalg-inner}\\
  &\phantom{= \sum} \ \ \ v(x,x')v(x,y')p(l,b)p(k,b')\nonumber\\
  &\phantom{= \sum} \ \ \ u(e_{gt\inv l\inv})[b\lact r(x)]e_j\# bx'\ot [b'\lact r(x)]u(e_{tk\inv})e_{j'}\#b'y'.\nonumber
\end{align}

Applying $(\ev_1\ot\id)^2$ to both equations we obtain
\begin{align}\label{eq:vp-coalg-1}
  \sum_{\substack{b\in B\\x'\in H}} p(g,b)v(x,x')bx'\ot bx' &= \sum_{t\in A} \gamma_x(gt\inv,t)p(e_{gt\inv})v(x)\ot p(e_t)v(x).
\end{align}
This is equivalent to
\begin{align}\label{eq:vp-coalg}
    \Delta(p(e_g)v(x)) = \left(\sum_{t\in A} p(\gamma_x(gt\inv,t)e_{gt\inv})\ot p(e_t)\right)\cdot v(x)\ot v(x).
\end{align}

Applying $(\id\ot\varepsilon)^2$ to \cref{eq:coalg-inner,eq:coalg-outer}, instead, we obtain
\begin{align}\label{eq:ur-comult1}
  \sum_{\substack{t\in A\\b\in B\\j,l,x'\in H}} & \theta'_j(b,x')\gamma'_{bx'}(jl\inv,l)p(t,b)v(x,x')u(gt\inv,j)r(x,j^b)
   e_{jl\inv}\ot e_l
\end{align}
and
\begin{align}\label{eq:ur-comult2}
\sum_{\substack{t\in G\\k,l\in A\\b,b'\in B\\j,j',x',y'\in H}} &\gamma_x(gt\inv,t)\theta'_j(b,x')\theta'_{j'}(b',y')
  u(gt\inv l\inv,j)u(tk\inv,j')\\
  & \ v(x,x')v(x,y')p(l,b)p(k,b') [b\lact r(x)]e_j\ot (b'.r(x))e_{j'}.\nonumber
\end{align}

\section{Properties of \ensuremath{v} and \ensuremath{u}}\label{sec:uv-properties}
Let $\psi=\morphquad\in\genhom$.  We wish now to consider properties of the components $v$ and $u$.  We will give a general picture for $v,u$ here, and in the next section we will specialize to the cases when $v,u$ are morphisms of (co)algebras.

We first have the following lemma.
\begin{lem}\label{lem:vinvert}
  Let $\morphquad\in\Hom(\D^\omega(G),\D^\eta(H))$.  Then $v(x)$ is invertible for all $x\in G$.
\end{lem}
\begin{proof}
  In $\D^\omega(G)$ we have \[(\varepsilon\# x)\inv = \sum_{g\in G}\theta_{xgx\inv}(x,x\inv)\inv e_g\# x\inv.\]  By \cref{eq:morphdef}, $\psi(\varepsilon\#x)=r(x)\# v(x)$.   Since $\ev_1\ot\id$ is a morphism of quasi-Hopf algebras by \cref{lem:obvious}, applying \cref{lem:biunital} shows that $v(x)$ has inverse \[\ev\ot\id \psi((\varepsilon\#x)\inv)=p(\sum_{g\in G}\theta_{xgx\inv}(x,x\inv)\inv e_g) v(x\inv).\]
\end{proof}

\begin{lem}\label{lem:vp-comm}
Let $\morphquad\in\genhom$.  Then
\begin{align}\label{eq:vp-comm}
v(x)p(e_g)v(x)\inv &= p(e_{xgx\inv})
\end{align}
for all $x,g\in G$.

As a consequence, the following all hold.
\begin{enumerate}
  \item $A$ is a normal subgroup of $G$.
  \item $A\subseteq Z(G) \iff p\comm v$.  In particular, if $B\subseteq Z(H)$ then $A\subseteq Z(G)$.
\end{enumerate}
\end{lem}
\begin{proof}
  Taking $x'=1$ and summing over $g$ in \cref{eq:pv-alg}, and then applying \cref{lem:vinvert} yields the desired identity.  We now consider the remaining claims in turn.

  \begin{enumerate}
    \item By invertibility of $v(x)$ we conclude that $p(e_g)\neq 0$ implies $p(e_{xgx\inv})\neq 0$ for all $x\in G$.  This is equivalent to the normality of $A$ by \cref{lem:p-desc}.
    \item The first statement follows immediately from the identity, and the second is then a trivial special case of the reverse direction.
   \end{enumerate}
\end{proof}
\begin{rem}
  Indeed, as was noted in \cref{rem:alt-alg}, $p$ and the normality of $A$ allows the (left) conjugation action of $G$ on $\du{A}$ to be pushed forward in the usual way to define a (left) $G$ action on $\kk B$: $x.p(e_a) = p(e_{xax\inv})$. Indeed, normality of $A$ means that$\widehat{A}$ is invariant under the $G$ action, which in turn means the action on $\kk B$ sends $B$ to itself.  As was stated in the aforementioned remark, this is explicitly given by $g\lact b = p(g\lact \chi_b)$.  By construction, $p$ then determines not only a Hopf algebra isomorphism $\du{A}\to\kk B$, but also an isomorphism of (left) $G$-modules.  The identity from the lemma further says that defining $x\star p(e_a) = v(x)p(e_a)v(x)\inv$ is not only a well-defined left $G$ action on $\kk B$, but that it is precisely equal to the $G$-module structure of the push forward.  This holds in spite of the fact that $v$ is not only not equal to the identity function in general, but is not even assumed to be an algebra or coalgebra morphism---see \cref{prop:v-is-alg,prop:v-is-coalg} for when $v$ has such properties.
\end{rem}

We also have the following, which is critical for isomorphisms.
\begin{cor}\label{cor:inj-surj-cent}
  Let $\morphquad\in\genhom$.  If $\morphquad$ is injective then $p\comm v$ and $A\subseteq Z(G)$.  On the other hand, if $\morphquad$ is surjective then $p\comm v$ $\iff$ $B\subseteq Z(H)$.  Finally, if $\morphquad$ is bijective then $p\comm v$, $A\subseteq Z(G)$, and $B\subseteq Z(H)$.
\end{cor}
\begin{proof}
  Apply \cref{thm:inj-surj} and the preceding lemma.
\end{proof}

Combining these results with \cref{df:p-quotients} we obtain the following general description of the algebra properties of the component $v$.
\begin{thm}\label{thm:v-extension-hopf}
  Let $\morphquad\in\genhom$.
    \begin{enumerate}
      \item \label{thm-part:pv-alg} The map $\D^\omega(G)\to \kk H$ defined by $e_g\# x \mapsto p(e_g)v(x)$ is a morphism of quasi-Hopf algebras.
      \item \label{thm-part:v-alg} The map $\kk B\#_p^\omega \kk G\to \kk H$ defined by $b\# g\mapsto bv(g)$ is a morphism of Hopf algebras.
    \end{enumerate}
\end{thm}
\begin{proof}
   We note that $\kk B\#_p^\omega\kk G$ is a Hopf algebra by \cref{lem:gen-omega-A,lem:p-embed-coassoc}.

   The first part follows from applying \cref{lem:obvious} to \[\ev_1\ot\id\circ\morphquad(e_g\# x) = p(e_g)v(x).\]

   The second part follows from \cref{eq:pv-alg} compared to the algebra structure of $\kk B\#_p^\omega \kk G$ given in \cref{df:p-quotients}.
\end{proof}

We conclude this section by obtaining a description for $u$ which is analogous to \cref{thm:v-extension-hopf}.

We recall the quasi-Hopf algebra $\du{H}\#_p^\eta\du{A}$ from \cref{df:p-embeddings}.

\begin{thm}\label{thm:u-extensions}
Let $\morphquad\in\genhom$.   Define $\pi_A \colon \du{G}\to\du{A}$ to be the canonical projection of Hopf algebras, which is the linear dual of the inclusion map $A\to G$.
    \begin{enumerate}
        \item The map $u\ot p\circ \Delta\colon\du{G}_\omega\to \D^\eta(H)$ is a morphism of quasi-Hopf algebras.

        \item The map $u\ot\pi_A \circ \Delta\colon\du{G}_\omega\to \du{H}\#_p^\eta \du{A}$ is a morphism of quasi-Hopf algebras.
    \end{enumerate}
\end{thm}
\begin{proof}
  The first part follows from \cref{lem:obvious} after observing that $u\ot p\circ\Delta$ is equivalent to the inclusion $\du{G}_\omega\to \D^\omega(G)$ followed by $\morphquad$.

  For the second part, by \cref{cor:p-quot-morph} we know that $\id\ot p$ yields an isomorphism of quasi-Hopf algebras $\du{H}\#_p^\eta\du{A}\to \Img(u\ot p\circ \Delta)$.  Letting $\id\ot p\inv$ denote its inverse function (which is the exact formula if we restrict the domain and codomain of $p$ so that it becomes an isomorphism $\du{A}\to \kk B$), we have that
  \[ u\ot \pi_A = \id\ot p\inv \circ (u\ot p\circ\Delta),\]
  and so $u\ot \pi_A$ is a morphism of quasi-Hopf algebras as desired.  Alternatively, this can be verified by using the quasi-Hopf algebra structures directly.
\end{proof}

\section{Bialgebra conditions for \ensuremath{v,u}}\label{sec:uv-bialg}
Now that we have found the fundamental properties of $v$ and $u$---though note that we have not yet considered relations between these two components---we will now consider when these components are themselves (co)algebra morphisms without the need to lift to an extension as in the preceding sections.

\begin{prop}\label{prop:v-is-alg}
  Let $\morphquad\in\genhom$.  Then the following are equivalent.
  \begin{enumerate}
    \item $v\colon\kk G\to\kk H$ is an algebra morphism.
    \item $\theta_a\equiv 1$ for all $a\in A$.
    \item $\kk B\#_p^\omega \kk G = \kk B\#_p^1\kk G$ as algebras.
  \end{enumerate}
\end{prop}
\begin{proof}
  We sum \cref{eq:pv-alg} over $g,g'$ and get \[ p(\sum_{a\in A} \theta_a(x,x')e_a)v(xx')=v(x)v(x').\]  That the second statement implies the first is then obvious.  That the first implies the second follows from this identity and \cref{lem:p-desc}.  The equivalence of the second and third statements follows by comparing the algebra laws of $\kk B\#_p^\omega\kk G$ and $\kk B\#_p^1\kk G$ given in \cref{df:p-embeddings}.
\end{proof}

\begin{prop}\label{prop:v-is-coalg}
  Let $\morphquad\in\genhom$.  Then the following are equivalent.
    \begin{enumerate}
        \item \label{prop-part:v-is-coalg-1} $v$ is a morphism of coalgebras.
        \item \label{prop-part:v-is-coalg-2} $\gamma_x\equiv 1$ on $A\times A$ for all $x\in G$.
        \item \label{prop-part:v-is-coalg-3} $\kk B\#_p^\omega \kk G = \kk B\#_p^1\kk G = \kk B \ot \kk G$ as coalgebras.
    \end{enumerate}
\end{prop}
\begin{proof}
  We refer to \cref{eq:vp-coalg}.  That \cref{prop-part:v-is-coalg-2} implies \cref{prop-part:v-is-coalg-1} follows immediately from this.  For the other direction, we note that
\[\sum_{g,t\in G}p(\gamma_x(gt\inv,t)e_{gt\inv})\ot p(e_t)\]
is invertible with inverse
\[\sum_{g,t\in G} p(\gamma_x(gt\inv,t)\inv e_{gt\inv})\ot p(e_t).\]
Now \cref{lem:vinvert} and \cref{prop-part:v-is-coalg-1} applied to \cref{eq:vp-coalg} forces
  \[\sum_{g,t\in G}\gamma_x(gt\inv,t)p(e_{gt\inv})\ot p(e_t) = 1\ot 1,\]
which is equivalent to $\gamma_x(a,b)=1$ for all $a,b\in A$ and $x\in G$ by \cref{lem:p-ids-1}, as desired.

  To complete the proof we now need only show that \cref{prop-part:v-is-coalg-2} is equivalent to \cref{prop-part:v-is-coalg-3}.  Note that by \cref{lem:p-embed-coassoc} that all three objects in \cref{prop-part:v-is-coalg-3} are Hopf algebras, and so are well-defined coassociative coalgebras. The equivalence then follows by comparing the coalgebra structures of $\kk B\#_p^\omega \kk G$, $\kk B\#_p^1\kk G$ given by \cref{df:p-embeddings}.
\end{proof}

Combining these results gives the following.
\begin{thm}\label{thm:v-is-bialg}
  Let $\morphquad\in\genhom$.  Then $v$ is a morphism of Hopf algebras if and only if $\theta\equiv 1$ on $A\times G\times G$ and $\gamma\equiv 1$ on $G\times A\times A$.
\end{thm}

We now focus on algebra and coalgebra conditions for the component $u$.

\begin{prop}\label{prop:u-is-alg}
  Let $\morphquad\in\genhom$.  Then the following are equivalent.
  \begin{enumerate}
    \item \label{prop-part:u-is-alg-2} $\sum_{b'\in B}p(t',b')\theta'_j(b,b') = \delta_{1,t'}$ for all $j\in H$, $t'\in G$, and $b'\in B$.
    \item \label{prop-part:u-is-alg-3} $\theta_j'\equiv 1$ on $B\times B$ for all $j\in H$.
    \item \label{prop-part:u-is-alg-4} $\du{H}\#_p^\eta \du{A} = \du{H}\#_p^1 \du{A}$ as algebras.
  \end{enumerate}
  Moreover, any of these conditions imply
    \begin{align}\label{eq:u-mult}
        u(fg) = u(f\com1)(p(f\com2)\lact u(g))
    \end{align}
    for all $f,g\in\du{G}$.  As a special case, if $B\subseteq Z(H)$ then this means that $u$ is an algebra morphism.

    As a partial converse, if $u$ is an algebra morphism and \cref{eq:u-mult} holds, then \cref{prop-part:u-is-alg-3} holds and $B\subseteq Z(H)$.
\end{prop}
\begin{proof}
  Take $x=x'=1$ in \cref{eq:ur-mult-1,eq:ur-mult-2} and use $v(1)=1$ to get
  \begin{align}\label{eq:u-alg-eq}
  u(e_g e_{g'})=\delta_{g,g'} u(e_g) = \sum_{\substack{t,t'\in A\\j\in H\\b,b'\in B}} p(t,b)p(t',b') \theta_j'(b,b') u(e_{gt\inv}) (b.u(e_{g't^{\prime -1}})) e_j.
  \end{align}
  By \cref{lem:biunital} for all $h\in H$ there exists $x\in G$ with $u(x,h)\neq 0$.  So by specializing to a suitable value of $g$ and taking $g'=g$ we can ensure that the coefficient on $e_j$ is non-zero whenever convenient.  We do this without further mention in the rest of the proof.

  We may rewrite the sum in \cref{eq:u-alg-eq} to
  \[ \sum_{\substack{t,t'\in A\\j\in H\\b\in B}} \Big(\sum_{b'\in B} p(t',b') \theta_j'(b,b')\Big) p(t,b) u(e_{gt\inv}) (b.u(e_{g't^{\prime -1}})) e_j.\]
  It immediately follows that \cref{prop-part:u-is-alg-2} implies \cref{eq:u-mult}.

  Now suppose that \cref{prop-part:u-is-alg-3} holds.  By \cref{eq:p-is-counital} the summation in \cref{eq:u-alg-eq} simplifies to $\sum u(e_{gt\inv})(p(e_t).u(e_{g'}))$, and so \cref{eq:u-mult} holds.  Furthermore the assumption makes \cref{prop-part:u-is-alg-2} equivalent to \cref{eq:p-is-counital}, which is known to hold.

  Next, suppose that \cref{prop-part:u-is-alg-2} holds.  Then considering the special case $t'=1$ gives \[ \frac{1}{|B|}\sum_{b'\in B}\theta_j'(b,b') = 1,\] and so \cref{prop-part:u-is-alg-3} holds by \cref{lem:average}.  This proves the equivalence of \cref{prop-part:u-is-alg-2,prop-part:u-is-alg-3}, and that each of these implies \cref{eq:u-mult}.

  The equivalence of \cref{prop-part:u-is-alg-3} and \cref{prop-part:u-is-alg-4} follows from comparing the algebra structures of $\du{H}\#_p^\eta \du{A}$ and $\du{H}\#_p^1 \du{A}$, which are well-defined by \cref{lem:vp-comm}.

  We now consider the partial converse.  Suppose that $u$ is an algebra morphism and that also \cref{eq:u-mult} holds.  We write out \cref{eq:u-mult} in the standard basis for $f=g=e_x$ for some $x\in G$ as
  \[ \sum_{j\in H} u(x,j) e_j = \sum_{\substack{j\in H\\t\in A\\b\in B}} u(xt\inv,j)p(t,b)u(x,j^b) e_j.\]
  We may then compare the coefficients on each $e_j$.  So fix $j\in H$ arbitrarily.  Since $u$ is an algebra morphism we may apply \cref{lem:u-simpler} and pick $x$ such that $u^*(j)=x$ and $u(x,s)\neq 0$ if and only if $s=j$.  Then in the right-hand side of the above equation the only non-zero terms contributing to the coefficient of $e_j$ occur when $t=1$, and so by \cref{lem:p-norms} we have for this $x,j$ that
  \[ e_j = \frac{1}{|A|}\sum_{\substack{b\in B\\ b\in C_H(j)}} e_j.\]
  Since $|A|=|B|$, this holds if and only if $B\subseteq C_H(j)$.  Since this depends only on $j$, which was arbitrary, we conclude that $B\subseteq Z(H)$.

  Now using that $B\subseteq Z(H)$, that $u$ is an algebra morphism, and \cref{lem:biunital}, for $g=g'$ and $x=x'=1$ the equality of \cref{eq:alg-inner,eq:alg-outer} simplifies to
  \begin{align*}
    \sum_{\substack{t\in A\\c\in B}} p(t,c)u(e_{gt\inv})\# c &= \sum_{\substack{j\in H\\t,t'\in A\\b,b'\in B}} p(t,b) p(t,b')\theta_j'(b,b')u(e_{gt\inv}) u(e_{gt^{\prime-1}}) e_j\# bb'\\
    &= \sum_{\substack{j\in H\\t\in A\\ b,c\in B}} \theta_j'(b,cb\inv) p(t,b)p(t,cb\inv)u(e_{gt\inv}) e_j\# c.
  \end{align*}
  Applying \cref{thm:p-simpler} to the $p(t,b)p(t,cb\inv)$ term in this last summation, we obtain
  \begin{align*}
    \sum_{\substack{j\in H\\t\in A\\c\in B}} p(t,c)u(e_{gt\inv})e_j\# c &= \sum_{\substack{j\in H\\t\in A\\ c\in B}} \left(\frac{1}{|A|}\sum_{b\in B} \theta_j'(b,cb\inv) \right) p(t,c)u(e_{gt\inv}) e_j\# c.
  \end{align*}
  Thus for all $j\in H$ and $c\in B$ we conclude that
  \[ \frac{1}{|A|}\sum_{b\in B} \theta_j'(b,cb\inv)=1,\]
  and so by \cref{lem:average} we see that \cref{prop-part:u-is-alg-3} holds.

   This completes the proof.
\end{proof}
The case where $\omega,\eta$ are trivial yields an improvement over the known characterizations of $u$ \citep{ABM,K14}.
\begin{cor}\label{cor:B-is-central}
  For $\morphquad\in \Hom(\D(G),\D(H))$, $u$ is a morphism of algebras if and only if $B\subseteq Z(H)$.
\end{cor}
\begin{proof}
  By \citep[Corollary 3.3]{ABM} we know that \cref{eq:u-mult} holds for all elements $\morphquad\in\Hom(\D(G),\D(H))$, so we need only apply the preceding proposition.
\end{proof}
The author remains unaware of any examples where $u$ is not an algebra morphism for $\omega,\eta$ trivial.

\begin{prop}\label{prop:u-is-coalg}
Let $\morphquad\in\genhom$.  Then the following are equivalent.
\begin{enumerate}
  \item \label{part:u-is-coalg-1} $\gamma'_b\equiv 1$ for all $b\in B$.
  \item \label{part:u-is-coalg-2} $\sum_{b\in B} \gamma'_b(j,k)p(t,b) = \delta_{1,t}$ for all $j,k\in H$ and $t\in A$.
  \item \label{part:u-is-coalg-3} $\du{H}\#_p^\eta \du{A}=\du{H}\#_p^1 \du{A}=\du{H}\ot\du{A}$ as coalgebras.
\end{enumerate}
Any of these conditions also implies that $u$ is a morphism of coalgebras.

As a partial converse, if $u$ is a morphism of Hopf algebras then $\gamma'_b\equiv 1$ for all $b\in B$.
\end{prop}
\begin{proof}
  Take $x=1$ in \cref{eq:ur-comult1} and \cref{eq:ur-comult2} to get that
  \begin{align}\label{eq:u-comult-a}
    \sum_{\substack{j,l\in H\\t\in A\\b\in B}} \gamma'_b(jl\inv,l)p(t,b)u(gt\inv,j) e_{jl\inv}\ot e_l
  \end{align}
  is equal to
  \begin{align}\label{eq:u-comult-b}
    \sum_{\substack{j,l\in H\\t\in A}} u(gt\inv,jl\inv)u(t,l) e_{jl\inv}\ot e_{l}= u\ot u\Delta e_g.
  \end{align}

  If either of \cref{part:u-is-coalg-1} or \cref{part:u-is-coalg-2} holds, then by \cref{lem:p-ids-1} the first equation simplifies to $\Delta u(e_g)$, and so $u$ is a morphism of coalgebras, as desired.

  We now proceed to show the equivalence of the first three statements.

  The equivalence of \cref{part:u-is-coalg-1} and \cref{part:u-is-coalg-3} follows from comparing the coalgebra structures of $\du{H}\#_p^\eta\du{A}$ and $\du{H}\#_p^1\du{A}$ given in \cref{df:p-quotients}

  Next, supposing that \cref{part:u-is-coalg-1} holds, by \cref{eq:p-is-counital} \[ \sum_{b\in B} \gamma'_b(j,k)p(t,b) = \sum_{b\in B}p(t,b)=\delta_{1,t},\]
  and so \cref{part:u-is-coalg-2} holds.  On the other hand, supposing that \cref{part:u-is-coalg-2} holds and specializing to the case $t=1$ we have
  \[ \frac{1}{|B|}\sum_{b\in B} \gamma'_b(j,k) = 1,\]
  so by \cref{lem:average} \cref{part:u-is-coalg-1} holds.

  Finally, for the partial converse, suppose that $u$ is a morphism of Hopf algebras.  Since $u$ is a morphism of coalgebras, the equality of \cref{eq:u-comult-a} to $\Delta(u(e_g))$ says that for all $j,l\in H$ and $g\in G$ we have
  \[u(g,j) = \sum_{\substack{b\in B\\t\in G}}\gamma'_b(jl\inv,l)p(t,b)u(gt\inv,j).\]
  Since $u$ is also a morphism of algebras, by \cref{lem:u-simpler} $u(g,j)=\delta_{g,u^*(j)}$ for all $g\in G$ and $j\in H$.  So for a given $j\in H$, we can consider the case $g=u^*(j)$.  In this case, the left-hand side is 1. On the other hand, in the right-hand side $u(gt\inv,j)=\delta_{u^*(j),gt\inv}$.  We conclude that the only non-zero term in the right-hand sum occurs when $t=1$, in which case by \cref{lem:p-norms} we have
  \[ 1 = \frac{1}{|B|}\sum_{b\in B}\gamma'_b(jl\inv,l)\]
  for all $j,l\in H$.  By \cref{lem:average}, $\gamma'_b\equiv 1$ for all $b\in B$ as desired.
\end{proof}

Combining these results, we obtain the following.
\begin{thm}\label{thm:u-is-bialg}
    Let $\morphquad\in\genhom$ have $B\subseteq Z(H)$.  Then $u$ is a morphism of Hopf algebras if and only if $\theta'\equiv 1$ on $H\times B\times B$ and $\gamma'\equiv 1$ on $B\times H\times H$.
\end{thm}
We observe that this is essentially a dual statement to \cref{thm:v-is-bialg}.

\section{Coalgebra conditions for \ensuremath{r}}\label{sec:r-coalg}

We have so far avoided discussing when the component $r$ has nice properties.  This is because we do not expect $r$ to have nice properties without some rather restrictive hypotheses, as was noted in the remarks following \cref{thm:purv}.  While $u,v$ needed conditions that depended only on the phases and $p$, good behavior for $r$ will also depend on $u,v$.  We will not attempt to extend $r$ in full generality as we extended $u,v$ in \cref{sec:uv-properties}.  Indeed, this does not seem possible beyond the assumption $\morphquad\in\genhom$.  However, under certain assumptions $r$ takes on a relatively simple description. In this section we derive some sufficient conditions for $r$ to be a morphism of coalgebras, and investigate some assumptions which make $r$ "almost" a coalgebra morphism, in a sense which will be made clear later. In the next section we carry out much the same considerations for the algebra case.

Our assumptions for both sections are tailored towards making the equations involving $r$ tractable.  The primary assumptions for this section and the next are that $u^*$ and $v$ are coalgebra morphisms.  This is because \cref{lem:coalg-to-group,lem:u-simpler} guarantee that these components have an easy description in terms of (identity-preserving) set maps $H\to G$ and $G\to H$ respectively.  If $u^*,v$ are morphisms of Hopf algebras, they are equivalent to group homomorphisms $H\to G$ and $G\to H$ respectively.  With either of these assumptions dropped, the results will no longer necessarily hold. Mostly this is because the phases prevent attempts to apply known identities to achieve any simplification, and as such the identities fail to collapse to anything much more concrete than what appears in \cref{sec:bialg-ids}.  Perhaps the simplest way to see this is to multiply $\varepsilon\# f\cdot \varepsilon \# f\inv$ in $\D^\omega(G)$, where $f$ is an invertible element of $\kk G$, and compare the cases of $f\in G$ and $f\not\in G$. In the first case the result is of the form $t\# 1$ for some $t\in \du{G}$, but in the second there may be a non-zero coefficient on a basis element of the form $e_s\# y$ for $s\in G$ and $1\neq y\in G$.

Our first lemma concerns a simple invertibility criterion.  Recall that an element $f\in\du{G}$ is invertible if and only if $f(g)\neq 0$ for all $g\in G$, which is in turn equivalent to every coefficient for $f$ in the standard basis of $\du{G}$ being non-zero.  The inverse is then obtained by simply inverting all coefficients in the standard basis.
\begin{lem}\label{lem:r-is-invertible}
  Let $\morphquad\in\genhom$ be such that $u^*,v$ are morphisms of coalgebras.  Then for all $x\in G$, $r(x)$ is invertible with inverse
  \[ r(x)\inv = \Big(\sum_{j\in H} \theta_j'(v(x),v(x\inv))\theta_{u^*(j)}(x,x\inv)\inv e_j\Big)[v(x)\lact r(x\inv)]. \]
  In particular, $v(x)\lact r(x\inv)=r(x)\inv$ if and only if
  \begin{align}\label{eq:theta-conv}
    \theta_{u^*(j)}(x,x\inv) = \theta_j'(v(x),v(x\inv))
  \end{align}
  for all $j\in H$.
\end{lem}
\begin{proof}
  We consider \cref{eq:ur-mult-1,eq:ur-mult-2} summed over $g,g'$ and with $x'=x\inv$ to get
  \[ \sum_{j\in H} \theta_{u^*(j)}(x,x\inv)e_j = \sum_{j\in H} \theta'_j(v(x),v(x\inv))r(x)[v(x)\lact r(x\inv)]e_j,\]
  which is equivalent to
  \[ \varepsilon = r(x)\Big(\sum_{j\in H} \theta_j'(v(x),v(x\inv))\theta_{u^*(j)}(x,x\inv)\inv e_j\Big)[v(x)\lact r(x\inv)].\]
  The claims now follow.
\end{proof}

We next establish the fundamental relation for $r$ to be a coalgebra morphism, provided that $u^*,v$ are also coalgebra morphisms.
\begin{thm}\label{thm:r-is-coalg}
  Let $\morphquad\in\genhom$, and suppose that $u^*,v$ are morphisms of coalgebras.  Then the following are equivalent
  \begin{enumerate}
    \item $r$ is a morphism of coalgebras.\label{part:r-coalg-2}
    \item For all $m,n\in H$, $k,l\in A$, and $x\in G$ \label{part:r-coalg-1}
    \begin{align}\label{eq:gamma-uv}
        \gamma_x(u^*(m)k,u^*(n)l) = \gamma'_{v(x)}(m,n).
    \end{align}
  \end{enumerate}
\end{thm}
\begin{proof}
  We rewrite \cref{eq:coalg-outer}, summed over $g\in G$, under the hypotheses as
  \begin{align}\label{eq:r-coalg-1}
    \sum_{j,l\in H} \gamma'_{v(x)}(jl\inv,l)r(x,j)e_{jl\inv}\# v(x) \ot e_l\# v(x),
  \end{align}
  and similarly we rewrite \cref{eq:coalg-inner} as
  \begin{align}\label{eq:r-coalg-2}
    \sum_{\substack{j,l\in H\\m,k\in A\\b,b'\in B}} \gamma_x(u^*(jl\inv)m,u^*(l)k)&p(m,b)p(k,b')\theta_{jl\inv}'(b,v(x))\theta'_l(b',v(x))\nonumber\\
    & r(x,b\inv jl\inv b) r(x,b^{\prime-1}lb') e_{jl\inv}\# b v(x) \ot e_l \# b'v(x).
  \end{align}
  Since $x$ is fixed, and $v(x)$ is invertible by \cref{lem:r-is-invertible}, then the equality of these two equations implies that only the terms with $b=b'=1$ in the latter equation have a non-zero coefficient.  So specializing to $b=b'=1$, using \cref{lem:p-norms}, and applying $\id\ot\varepsilon$ to both equations, yields the equality of
  \begin{align}\label{eq:r-coalg-3}
    \sum_{j,l\in H} \gamma'_{v(x)}(jl\inv,l)r(x,j)e_{jl\inv} \ot e_l = \Big(\sum_{m,n\in H} \gamma'_{v(x)}(m,n)e_m\ot e_n\Big)\Delta(r(x))
  \end{align}
  and
  \begin{align}
    \frac{1}{|A|^2}\sum_{\substack{j,l\in H\\ a,a'\in A}} &\gamma_x(u^*(jl\inv)a,u^*(l)a')r(x,jl\inv) r(x,l) e_{jl\inv}\ot e_l\nonumber\\
      =& \Big(\frac{1}{|A|^2} \sum_{a,a'\in A} \gamma_x(u^*(m)a,u^*(n)a')e_m\ot e_n\Big) r(x)\ot r(x).\label{eq:r-coalg-4}
  \end{align}

  The equivalence of \cref{part:r-coalg-1,part:r-coalg-2} then follows from \cref{lem:average,lem:r-is-invertible}.
\end{proof}
The following corollary is a special case of the theorem we wish to single out.
\begin{cor}\label{cor:gamma-cosets}
  Let $\morphquad\in\genhom$.  If $u^*,v,r$ are all morphisms of coalgebras, then
  \begin{align}\label{eq:gamma-cosets}
    \gamma_x(u^*(m)a,u^*(n)b) = \gamma_x(u^*(m),u^*(n))
  \end{align}
  for all $a,b\in A$, $x\in G$, and $m,n\in H$.
\end{cor}

In the case when $\omega\in Z^3(G/A,U(1))$---which can be forced in certain situations by \cref{thm:gen-omega-u}, and in fact forces $v$ to be a morphism of Hopf algebras by \cref{thm:v-is-bialg}---we see that \cref{eq:gamma-cosets} is trivially satisfied.  Since we have seen the necessity of the condition (assuming $u^*,v$ are morphisms of coalgebras), it is natural to wonder how close it is to being a sufficient condition for $r$ to be a morphism of coalgebras.  The rest of this section is dedicated to this question.

\begin{prop}\label{prop:r-quasi-coalg}
  Let $\morphquad\in\genhom$.  Suppose that $v,u^*$ are morphisms of coalgebras.  Fix $x\in G$, and suppose further that \cref{eq:gamma-cosets} holds for this $x$ and all $m,n\in H$, $a,b\in A$.  Then for all $m,n\in H$ we have
  \begin{align}\label{eq:r-quasi-coalg}
    r(x,mn) = \frac{\gamma_x(u^*(m),u^*(n))}{\gamma'_{v(x)}(m,n)}\,r(x,m)r(x,n).
  \end{align}
\end{prop}
\begin{proof}
  Apply the assumptions to further simplify \cref{eq:r-coalg-3,eq:r-coalg-4} and compare coefficients to obtain the desired equality.
\end{proof}

One may interpret the result as saying that, under the hypotheses, $r(x)$ is a quasi-projective one-dimensional representation of $H$, in the following sense.  For fixed $x\in G$ let $\omega^x$ be the 3-cocycle obtained from $\omega$ by the right conjugation action of $x$ on all inputs.  By \cref{lem:gen-u-omega-eta} we can also define $\eta^x$ by
    \[\eta^x(m,n,k)=\omega^x(u^*(m),u^*(n),u^*(k)).\]
Now define a 3-cocycle $\xi$ of $H$ by
   \[ \xi_x(m,n,k) = \frac{\eta^{v(x)}(m,n,k)}{\eta^x(m,n,k)} = \frac{\omega(u^*(m^{v(x)}),u^*(n^{v(x)}),u^*(k^{v(x)}))}{\omega(u^*(m)^x,u^*(n)^x,u^*(k)^x)}.\]
Next define
   \[ \kappa_x(m,n) = \frac{\gamma_x(u^*(m),u^*(n))}{\gamma'_{v(x)}(m,n)}.\]
for all $m,n\in H$ and $x\in G$.  Applying \cref{eq:coassoc} we obtain the relation
   \[ \kappa_x(m,n)\kappa_x(mn,k) = \kappa_x(n,k)\kappa_x(m,nk) \xi_x(m,n,k).\]
Moreover, it is well known that $\Inn(G)$ acts trivially on the (co)homology of $G$, so we conclude that $\xi_x$ is cohomologically trivial.  Thus $\kappa_x$ is a 2-cocycle of $H$ up to the cohomologically trivial 3-cocycle $\xi_x$, for all $x\in G$.  When $\xi_x$ is precisely the trivial 3-cocycle then $\kappa_x$ is exactly a 2-cocycle, and so $r(x)$ is a projective one-dimensional representation of $H$.

Note that if we have $u^*(m^{v(x)}) = u^*(m)^x$ for all $x\in G$ and $m\in H$ then $\xi_x$ is (trivially) trivial.  Our goal now will be to find conditions that assure this equality holds.  Recalling that our overall goal is to find quasi-bialgebra isomorphisms $\D^\omega(G)\cong\D^\eta(H)$ which also define Hopf algebra isomorphisms $\D(G)\cong D(H)$, we first establish the following lemma, which yields some identities that we use as assumptions in other results.

\begin{lem}\label{lem:untwist-cond}
Let $\morphquad\in\genhom\cap \Hom(\D(G),\D(H))$ have $u$ Hopf.  Then $\theta'\equiv 1$ on $H\times B\times\Img(v)$ and $H\times\Img(v)\times B$.
\end{lem}
\begin{proof}
  By \citep{ABM,K14} we know that $u$ is necessarily a morphism of coalgebras, and $r,v$ are necessarily morphisms of Hopf algebras.   Thus $\Img(v)$ can naturally be identified with a subgroup of $H$.  Now by \cref{eq:morphdef} we have
  \begin{align*}
   \morphquad(e_g\# x) &= \sum_{\substack{j,l\in H\\k\in A\\b\in B}} \left(u(e_{gk\inv})e_j\#p(k,b)b\right)\cdot \left(r(x)e_l\# v(x)\right)\\
   &= \sum_{\substack{j\in H\\k\in A\\b\in B}} \theta'_j(b,v(x))  u(e_{gk\inv})[b\lact r(x)] e_j \# p(k,b)bv(x)\\
   &= \sum_{\substack{j\in H\\k\in A\\b\in B}} \theta'_j(b,v(x))  u(e_{gk\inv})r(x) e_j \# p(k,b)bv(x)\\
   &= \sum_{\substack{j\in H\\k\in A\\b\in B}} u(e_{gk\inv})r(x)e_j\# p(k,b)bv(x),
  \end{align*}
  where either \cref{lem:coalg-to-group} or \cref{cor:B-is-central} gives the third equality and $\morphquad\in\Hom(\D(G),\D(H))$ gives the last equality.  By \cref{lem:u-simpler,lem:r-is-invertible} we conclude that for any fixed $j\in H$ and $x\in G$  we can pick $g\in G$ with $u^*(j)=g$ (forcing $k=1$ for a non-zero contribution in both of the last summations for that $j$), and guarantee that the coefficient on $e_j\#bv(x)$ is non-zero.  By \cref{lem:p-norms}, on the one hand this coefficient is $r(x,j)/|B|$, and on the other it is $\theta_j'(b,v(x)) r(x,j)/|B|$.  We have already noted that $r(x,j)\neq 0$, so we conclude that $\theta_j'(b,v(x))=1$.  The arbitrariness of $j,b,x$ then gives the first half of the claim.

  On the other hand, for $\morphquad\in\Hom(\D(G),\D(H))$ by \citep[Corollary 2.3]{K14} we have the relation
  \begin{align}\label{eq:uv-rel}
    v(x)\rightharpoonup u(e_g) = u(e_{xgx\inv})
  \end{align}
  for all $x,g\in G$; moreover, this identity is equivalent to $u^*(m^{v(x)}) = u^*(m)^x$ for all $x\in G$ and $m\in H$, which was   mentioned earlier.  Considering \cref{eq:ur-mult-1,eq:ur-mult-2} under the assumptions, and summing over $g\in G$ and setting $x'=1$ by \cref{lem:biunital} we get
  \[ u(e_{xg'x\inv}) = \sum_{\substack{j\in H\\t\in A\\b\in B}} \theta'_j(v(x),b)p(t,b)(v(x)\rightharpoonup u(e_{g't\inv}))e_j.\]
  Expanding in the basis elements and using that $u$ is a morphism of Hopf algebras we have
  \[ \sum_{\substack{l\in H\\u^*(l)=g'}}e_{v(x) l v(x\inv)} = \sum_{\substack{h\in H\\t\in A\\b\in B\\u^*(h)=g't\inv}}\theta'_{v(x) h v(x\inv)}(v(x),b)p(t,b) e_{v(x) h v(x\inv)}.\]
  We conclude that we must have $\sum_b \theta'_j(v(x),b)p(t,b) = \delta_{1,t}$ for all $x\in G,j\in H$, so by \cref{lem:average,lem:p-norms} applied to the special case $t=1$ we get $\theta'_j(v(x),b)=1$ for all $x\in G,b\in B,j\in H$, as desired.

  This concludes the proof.
\end{proof}
\begin{cor}\label{cor:swap-ab-decomp}
If $\psi=\morphquad\in\genhom\cap\Hom(\D(G),\D(H))$ has $u$ Hopf then
\[ \psi(e_{xgx\inv}\# x) = \psi(\varepsilon\# x)\cdot \psi(e_g\# 1),\]
where the multiplication can be carried out in either $\D(H)$ or $\D^\eta(H)$.
\end{cor}
\begin{proof}
  Either use that $\theta_j'\equiv 1$ on $\Img(v)\times B$ or use that $\psi$ is an algebra morphism of both the twisted and untwisted doubles.
\end{proof}

The next lemma shows that these triviality conditions are sufficient to deduce a certain cancellation law for the codomain's multiplicative phase.
\begin{lem}\label{lem:theta-quot}
  Let $X\subseteq H$ be a subgroup and $B\subseteq H$ a subgroup which is closed under conjugation by elements in $X$.  Let also $\eta\in Z^3(H,U(1))$ with associated multiplicative phase $\theta$ as usual. Then if $\theta$ is trivial on all three of $H\times B\times X$, $H\times X\times B$, and $H\times B\times B$, then $\theta_j(xb,yc) =\theta_j(x,y)$ for all $x,y\in X$ and $b,c\in B$.
\end{lem}
\begin{proof}
   We will make frequent use of \cref{eq:assoc}, where changing the order of the multiplications can result in distinct identities.  Throughout $j\in H$, $x,y\in X$ and $b,c\in B$.

  Note that closure of $B$ with respect to conjugation by $X$ means that $by=yb'$, where $b'=y\inv b y\in B$.

  First we deduce from the fact that $\theta$ vanishes on $H\times X\times B$ and $H\times B\times B$ and \cref{eq:assoc} that
  \begin{align}\label{eq:theta-quot-1} \theta_j(xb,c) = \theta_j(x,b)\inv \theta_j(x,bc)\theta_{j^x}(b,c) = 1,\end{align}
  as desired.  By \cref{eq:assoc} and the assumption that $\theta$ vanishes on $H\times X\times B$ and $H\times B\times X$ we also have
  \begin{align}\label{eq:theta-quot-2} \theta_j(xb,y) = \theta_j(x,b)\inv \theta_j(x,by) \theta_{j^x}(b,y) =\theta_j(x,by).\end{align}

  Since $X$ is closed under multiplication, and by the assumptions that $\theta$ vanishes on $H\times X\times B$, a third use of \cref{eq:assoc} gives
  \begin{align}\label{eq:theta-quot-3}\theta_j(x,yb) = \theta_j(x,y) \theta_j(xy,b) \theta_{j^x}(y,b)\inv=\theta_j(x,y).\end{align}

  By \cref{eq:assoc} and the assumption that $\theta$ vanishes on $H\times B\times X$ and $H\times X\times B$, followed by \cref{eq:theta-quot-1} and the closure of $B$ under conjugation by $X$, we find
  \begin{align}\label{eq:theta-quot-4}
    \theta_j(b,yc) = \theta_j(b,y)\theta_j(by,c)\theta_{j^b}(y,c)\inv = \theta_j(by,c) = 1.
  \end{align}

  Finally, using that $B$ is closed under conjugation by elements of $X$, that $\theta$ vanishes on $X\times X\times B$, \cref{eq:theta-quot-3,eq:theta-quot-4}, and another application of \cref{eq:assoc} we get
  \begin{align*}
    \theta_j(xb,yc) =& \theta_j(x,b)\inv \theta_j(x,byc) \theta_{j^x}(b,yc) = \theta_j(x,y),
  \end{align*}
  as desired.
\end{proof}

The assumption $\morphquad\in\Hom(\D(G),\D(H))$ automatically forces $r$ to be a morphism of Hopf algebras, and we'd like to know if there are less restrictive conditions to ensure triviality of $\xi_x$.

\begin{lem}
  Let $x\in G$ and $\morphquad\in\genhom$ satisfy all of the following.
  \begin{enumerate}
    \item $u,v$ are morphisms of Hopf algebras.
    \item $b\lact r(x)=r(x)$ for all $b\in B$.
    \item $\theta'_j(b,v(x))=\theta'_j(v(x),b)=1$ for all $b\in B$.
  \end{enumerate}
  Then $u,v$ satisfy \cref{eq:uv-rel} for this $x$ and all $g\in G$.
\end{lem}
\begin{proof}
    Summing \cref{eq:ur-mult-1,eq:ur-mult-2} over $g$ and setting $x'=1$ under the assumptions yields
    \[  [v(x)\lact u(e_{g'})] r(x) = u(e_{xg'x\inv}) r(x).\]
    Since $r(x)$ is invertible by \cref{lem:r-is-invertible}, we may cancel it and thereby recover \cref{eq:uv-rel} for this $x$, as desired.
\end{proof}
\begin{rem}
Recall that \cref{prop:u-is-alg} provided conditions which guaranteed that $B\subseteq Z(H)$, which would make the second condition trivially satisfied for all $x\in G$.
\end{rem}
This yields the following triviality condition for $\xi_x$.

\begin{thm}
  Let $\morphquad\in\genhom$ and $x\in G$ satisfy all of the following.
  \begin{enumerate}
    \item $u,v$ are morphisms of Hopf algebras;
    \item \cref{eq:gamma-cosets} holds for this $x$ and all $a,b\in A$, $m,n\in H$;
    \item $B$ acts trivially on $r(x)$;
    \item $\theta'$ vanishes on $H\times \Img(v)\times B$ and $H\times B\times \Img(v)$.
  \end{enumerate}
  Then $\xi_x$ is trivial and $\kappa_x$ is a 2-cocycle of $H$.  Thus, $r(x)$ is a projective one-dimensional representation of $H$.
\end{thm}
\begin{proof}
  Apply the preceding lemma and discussions.
\end{proof}

Indeed, under the hypotheses above
 \[ \kappa_x(m,n) = \frac{\gamma_x(u^*(m),u^*(n))}{\gamma_{u^*v(x)}(u^*(m),u^*(n))}\]
is itself very nearly trivial, in the sense that by \citep[Lemma 3.10]{K14} \cref{eq:uv-rel} is equivalent to
\begin{align}\label{eq:uv-cent}
u^*v(x)x\inv \in C_G(\Img(u^*)).
\end{align}

\section{Algebra conditions for \ensuremath{r}}\label{sec:r-alg}

We now consider a variety of conditions for when $r$ is a morphism of algebras.  Later in the section we will also consider how close $r(\cdot,j)$ is to a projective linear character of $G$.

\begin{thm}\label{thm:r-is-alg}
    Let $\morphquad\in\genhom$, and suppose all of the following hold.
    \begin{enumerate}
      \item $u^*,v,r$ are morphisms of coalgebras.
      \item $\theta'_j(b,v(x))=1$ for all $b\in B,\,h\in H,\,x\in G$.
    \end{enumerate}
    Then $r$ is a morphism of algebras if and only if
    \begin{align}\label{eq:theta-uv}
        \theta_{u^*(h)}(x,y) = \theta'_h(v(x),v(y))
    \end{align}
    for all $h\in H$ and $x,y\in G$.
\end{thm}
\begin{proof}
  Applying the assumptions to \cref{eq:ur-mult-1,eq:ur-mult-2} and summing over $g,g'\in G$ we get
  \[ \Big(\sum_{h\in H} \theta_{u^*(h)}(x,x')e_h\Big) r(xx') = \Big(\sum_{j\in H} \theta'_j(v(x),v(x')) e_j\Big) r(x)r(x').\]
  It is easily seen that the terms in parentheses are both units, and the desired equivalence then follows.
\end{proof}

\begin{thm}\label{thm:r-twist-alg}
  Let $\morphquad\in\genhom$ have $u^*$ a morphism of coalgebras and $v$ a morphism of Hopf algebras.  Then $r$ satisfies the identity
  \begin{align}\label{eq:r-twist-alg}
    \Big( \frac{1}{|A|} \sum_{\substack{t\in A\\ j\in H}} \theta_{u^*(j)t}(x,y) e_j\Big) r(xy) = \Big(\sum_{j\in h} \theta'_j(v(x),v(y)) e_j\Big) r(x)[v(x)\lact r(y)]
  \end{align}
  for all $x,y\in G$.
\end{thm}
\begin{proof}
  We sum \cref{eq:alg-inner,eq:alg-outer} over $g,g'$ to get
  \begin{align*}
    \sum_{j\in H} r(x)&[v(x)\lact r(x')]\theta_j'(v(x),v(x')) e_j\# v(xx')\\
     &= \sum _{\substack{j\in H\\t\in A\\b\in B}}\theta_{u^*(j)t}(x,x')\theta_j'(b,v(xx')) p(t,b) [b\lact r(xx')]e_j\# bv(xx').
  \end{align*}
  Since $x,x'$ are fixed, equality means that for each $j\in H$ all terms in the right-hand sum with fixed $b\neq 1$ vanish.  Thus the above equality reduces to
  \begin{align}
    \sum r(x)&[v(x)\lact r(x')]\theta_j'(v(x),v(x')) e_j\# v(xx')\nonumber\\
     &= \Big(\frac{1}{|A|}\sum \theta_{u^*(j)t}(x,x') e_j\Big)r(xx')\# v(xx').
  \end{align}
  Applying $\id\ot\varepsilon$ then yields the desired relation.
\end{proof}

We note that by \cref{lem:r-is-invertible} the assumptions guarantee that $r(x)$ and $v(x)\lact r(y)$ are invertible for all $x,y\in G$, from which it follows that
  \[ \frac{1}{|A|} \sum_{\substack{t\in A\\ j\in H}} \theta_{u^*(j)t}(x,y) e_j \]
is invertible for all $x,y\in G$.

This gives the following.
\begin{cor}
  Let $\morphquad\in\genhom$ have $u^*$ a morphism of coalgebras and $v$ a morphism of Hopf algebras.  Then $r$ satisfies
  \[ r(xy) = r(x)[v(x)\lact r(y)]\]
  for $x,y\in G$ if and only if
  \begin{align}\label{eq:theta-conversion}
    \theta_{u^*(j)t}(x,y) = \theta'_j(v(x),v(y)) , \ \forall t\in A \mbox{ and } j\in H.
  \end{align}
\end{cor}
A special case of \cref{eq:theta-conversion} is
\begin{align}\label{eq:theta-cosets}
    \theta_{u^*(j)t}(x,y) = \theta_{u^*(j)}(x,y) \ \forall t\in A \mbox{ and } j\in H.
\end{align}
As in the previous section, we then naturally consider what this equation says about $r(\cdot,j)$, and in particular wonder how closely $r(\cdot,j)$ resembles a projective linear character of $G$.

\begin{prop}
  Let $\morphquad\in\genhom$ have $u^*$ a morphism of coalgebras and $v$ a morphism of Hopf algebras.  Fix $j\in H$ and suppose that for all $x,y\in G$ and $t\in A$ that \cref{eq:theta-cosets} holds.  Then
  \[ r(xy,j) = \frac{\theta'_j(v(x),v(y))}{\theta_{u^*(j)}(x,y)} r(x,j) r(y,j^{v(x)}).\]
\end{prop}
\begin{proof}
  Apply the assumptions to \cref{thm:r-twist-alg}.
\end{proof}
Note that \cref{eq:assoc} implies, under the assumptions of the Proposition, that $r((xy)z,j)=r(x(yz),j)$.  If also $r(y,j)=r(y,j^{v(x)})$ for all $x,y\in G$---which is assured if $r$ is a morphism of coalgebras---, then as a special case we see that $r(\cdot,j)=\ev_j\circ r$ is a projective linear character of $G$.

\section{Preservation of antipodes}\label{sec:antipode}
Since $\genhom\neq\genhomh$ unless $\omega=\eta=1$ in general, we wish to find conditions on elements  $\morphquad\in\genhom\cap\Hom(\D(G),\D(H))$ that guarantee $\morphquad\in\genhomh$.

First, we show that the $\beta$ element is always preserved (remember that the $\alpha$ elements are always the identity here).
\begin{lem}\label{lem:beta-pres}
  Let $\morphquad\in\genhom$ have $u$ a morphism of Hopf algebras.  Then $\morphquad(\beta)=\beta'$.
\end{lem}
\begin{proof}
Let $\psi=\morphquad$.  Then
\begin{align*}
  \psi(\phi) &= \sum_{g\in G}\omega(g,g\inv,g) \psi(e_g\#1)\\
  &= \sum_{\substack{g\in G,t\in A}} \omega(g,g\inv,g) u(e_{gt\inv})\# p(e_t)\\
  &= \sum_{\substack{j\in H\\t\in A}} \omega(u^*(j)t,(u^*(j)t)\inv,u^*(j)t) e_j\# p(e_t),
\end{align*}
where the last equality follows from \cref{lem:u-simpler}.  Now applying \cref{lem:vp-comm,thm:gen-omega-u,lem:p-ids-1} to this last summation we get
\begin{align}
  \psi(\phi) &= \sum_{j\in H} \omega(u^*(j),u^*(j\inv),u^*(j)) e_j\# 1.
\end{align}
By \cref{lem:gen-u-omega-eta} this last term is precisely $\beta'$.
\end{proof}
Before considering the antipode relation, we will need the following identities.
\begin{prop}\label{prop:thetaprime-bv-vanish}
    Let $\morphquad\in\genhom\cap\Hom(D(G),D(H))$ have $u$ Hopf.  Then for all $b,c\in B$, $g,h,k\in H$, and $x,y\in G$ the following all hold.
    \begin{enumerate}
      \item $\theta'_g(bv(x),cv(y))=\theta'_g(v(x)b,cv(y))=\theta'_g(bv(x),v(y)c)=\theta'_g(v(x),v(y))$.
      \item $\eta(b\inv g b, b\inv h b, b\inv k b) = \eta(g,h,k)$.\label{prop-part:vanish-2}
      \item $\gamma'_{bv(x)}(g,h)= \gamma'_{v(x)}(g,h)$.
    \end{enumerate}
\end{prop}
\begin{proof}
  By \citep[Corollary 3.3]{ABM} $v$ is a morphism of Hopf algebras.

  Now \cref{eq:vp-comm} means that $B$ is closed under conjugation by elements of the subgroup $\Img(v)$.  Therefore the first part follows from \cref{prop:u-is-alg,lem:untwist-cond,lem:theta-quot}.

  The second part is trivially true by \cref{cor:B-is-central}; it also follows from \cref{eq:coassoc,prop:u-is-coalg} without knowing $B\subseteq Z(H)$.

  The third part follows from \cref{eq:struc-compat,lem:untwist-cond,prop:u-is-coalg,prop-part:vanish-2}.
\end{proof}

\begin{thm}\label{thm:QB-is-hopf}
  Let $\morphquad\in\genhom\cap\Hom(\D(G),\D(H))$ have $u$ a morphism of Hopf algebras.  Then $\morphquad\in\genhomh$ if and only if \[ \theta_{ag\inv}(x,x\inv)\gamma_x(ga\inv,ag\inv) = \theta_g(x,x\inv)\gamma_x(g,g\inv)\] for all $a\in A$ and $g,x\in G$.
\end{thm}
\begin{proof}
By \cref{lem:beta-pres} we need only show that $S\circ \morphquad = \morphquad\circ S$.

Let $\psi=\morphquad\in\genhom\cap\Hom(D(G),D(H))$.  Note that by \citep[Theorem 2.1 and Corollary 2.3]{K14} the assumptions imply that $p,u,r,v$ are all morphisms of Hopf algebras.  Then we have
\begin{align}
  \psi(S(e_g\# x)) =&\ \theta_{g\inv}(x,x\inv)\inv \gamma_x(g,g\inv)\inv \psi(S_{\D(G)}(e_g\# x))\label{eq:psi-S},
\end{align}
and by \cref{lem:u-simpler}
\begin{align}
  S(\psi(e_g\# x)) =& S( \sum r(x) e_j \# p(a,b)b v(x) )\nonumber\\
    =& \sum\label{eq:S-psi}
            \theta'_{j\inv }(bv(x),(bv(x))\inv)\inv \gamma'_{bv(x)}(j,j\inv)\inv\\
             &\qquad S_{\D(H)}(r(x) e_j \# p(a,b)b v(x)\inv),\nonumber
\end{align}
where these last two sums are over all $j\in H$, $a\in A$, and $b\in B$ satisfying $u^*(j)=ga\inv$.  The statement that $\psi$ is a morphism of quasi-Hopf algebras is then equivalent to \eqref{eq:psi-S}=\eqref{eq:S-psi}.

Let us consider the final sum in \cref{eq:S-psi}.  By \cref{prop:thetaprime-bv-vanish}
\[ \theta'_{j\inv}(bv(x),(bv(x))\inv) = \theta'_{j\inv}(v(x),v(x)\inv),\]
so using $u^*(j)=ga\inv$ and \cref{thm:r-is-alg} we have
\[ \theta'_{j\inv}(bv(x),(bv(x))\inv) = \theta_{ag\inv}(x,x\inv).\]

Continuing to apply \cref{prop:thetaprime-bv-vanish}, we also have
\[\gamma'_{bv(x)}(j,j\inv) = \gamma'_{v(x)}(j,j\inv),\]
and so for $u^*(j)=ga\inv$ by \cref{thm:r-is-coalg} we get
\[\gamma'_{bv(x)}(j,j\inv) = \gamma_x(ga\inv,ag\inv).\]

Now if we have
\[ \theta_{ag\inv}(x,x\inv)\gamma_x(ga\inv,ag\inv) = \theta_g(x,x\inv)\gamma_x(g,g\inv)\]
for all $a\in A$ and $g,x\in G$, it then follows that $\psi\in\genhomh$ if and only if $\psi S_{\D(G)} = S_{\D(H)}\psi$, and this is guaranteed by the assumption $\psi\in\Hom(\D(G),\D(H))$.

On the other hand suppose \eqref{eq:psi-S}=\eqref{eq:S-psi}. By \cref{cor:swap-ab-decomp}, \cref{eq:psi-S} is equal to
\begin{align}\label{eq:psi-S2}
    \theta_{g\inv}(x,x\inv)\inv \gamma_x(g,g\inv)\inv \sum_{\mathclap{u^*(k)=x\inv g\inv a\inv x}}r(x\inv,k)e_k\# p(a,b)v(x\inv)b,
\end{align}
and we may rewrite \cref{eq:S-psi} as
\begin{align}\label{eq:S-psi2}
    \sum \theta_{ag\inv}(x,x\inv)\inv \gamma_x(ga\inv,ag\inv)\inv r(x\inv,k) e_k\# p(a\inv,b)v(x\inv)b.
\end{align}

Equality then means that for any fixed $k\in H$ and $x\in G$, we may consider the $b=1$ terms.  We obtain
\[ \frac{1}{|A|} \sum_{a\in A} \theta_{ag\inv}(x,x\inv)\inv \gamma_x(ga\inv,ag\inv)\inv = \theta_{g\inv}(x,x\inv)\inv \gamma_x(g,g\inv)\inv.\]
The desired equality now follows from \cref{lem:average}.

This completes the proof.
\end{proof}

\begin{cor}\label{cor:qb-is-qh}
  If $\morphquad\in\genhom\cap\Hom(\D(G),\D(H))$ has $u$ Hopf and $A\Img(u^*)=G$ then \[\morphquad\in\genhomh.\]
\end{cor}
\begin{proof}
 By \cref{thm:gen-omega-u} we have $\omega\in Z^3(G/A,U(1))$, from which it follows that the phases $\theta,\gamma$ are also well-defined on $G/A$.  We then get the result by applying the preceding theorem.
\end{proof}

\section{Rigid Isomorphisms}\label{sec:rigid-isoms}
We are now ready to establish the main results of this paper.
\begin{df}\label{df:rigid-isoms}
  We say a quasi-bialgebra isomorphism \[\morphquad\colon \D^\omega(G)\to\D^\eta(H)\] is rigid if $\morphquad\in\Hom(\D(G),\D(H))$. We denote the set of all such isomorphisms by $\genisomr$.  We also denote $\Isom_{\operatorname{rigid}}(\D^\omega(G),\D^\omega(G))$ by $\genautr$.
\end{df}
The definition of $\genisomr$ means that $\morphquad$ is also a Hopf algebra isomorphism $\D(G)\cong\D(H)$. It is easy to verify that $\genautr$ forms a subgroup of $\Aut(\D^\omega(G))$, where the composition can be computed in $\Aut(D(G))$.

\begin{cor} If $\genisomr\neq\emptyset$ then $G\cong H$.  Furthermore, if $\morphquad\in\genisomr$ then $\morphquad\in \genhomh$.
\end{cor}
\begin{proof}
  The first statement follows from \citep[Theorem 3.5]{K14}.  The second follows from \citep[Corollary 3.3]{ABM}, \cref{thm:inj-surj,cor:inj-surj-cent,cor:B-is-central,cor:qb-is-qh}
\end{proof}
The last claim is worth restating in words: the rigid quasi-bialgebra isomorphisms are in fact quasi-Hopf algebra isomorphisms.  As such we do not need to introduce the (obvious) notion of  rigid quasi-Hopf algebra isomorphism. The first claim also says that we lose no generality by restricting focus to $\genisomrg$.

\begin{thm}\label{thm:rigid-isom-equiv}
  Let $\morphquad\in\genhomg$ be an isomorphism.  Then $\morphquad$ is rigid if and only if the following all hold.
  \begin{enumerate}
    \item $\omega\in Z^3(G/A,U(1))$.\label{item:omega-mod-A}
    \item $\gamma_g(u^*(x),u^*(y)) = \gamma'_{v(g)}(x,y)$ for all $g\in G$ and $x,y\in H$.\label{item:gamma-compat}
    \item $\theta_{u^*(x)}(g,h) = \theta'_x(v(g),v(h))$ for all $g,h\in G$ and $x\in H$.\label{item:theta-compat}
    \item $\theta'_x(y,z)=1$ for all $x\in H$ and $(y,z)\in B\times\Img(v)\cup \Img(v)\times B\cup B\times B$.\label{item:theta-B}
    \item $\gamma'_b(x,y)=1$ for all $b\in B$ and $x,y\in H$.\label{item:gamma-B}
  \end{enumerate}
\end{thm}
\begin{proof}
  For the forward direction, combine all of \cref{lem:gen-u-omega-eta,thm:gen-omega-u,cor:inj-surj-cent,thm:u-is-bialg,thm:v-is-bialg,thm:r-is-coalg,thm:r-is-alg,lem:theta-quot,lem:untwist-cond} with \citep[Theorem 2.1 and Corollary 2.3]{K14}.

  For the reverse direction, we suppose that \cref{item:omega-mod-A,item:gamma-compat,item:theta-compat,item:theta-B,item:gamma-B} all hold.  Our goal is to show that all components are morphisms of Hopf algebras, and to show that the identities in \cref{sec:bialg-ids} yield the defining identities for $\End(\D(G))$ from \citep[Theorem 2.1 and Corollary 2.3]{K14}.

  From the assumption that $\morphquad$ is an isomorphism, by \cref{cor:inj-surj-cent} we conclude that $A,B\subseteq Z(G)$.  From \cref{item:omega-mod-A} applied to \cref{thm:v-is-bialg} we conclude that $v$ is a morphism of Hopf algebras.  From \cref{item:theta-B,item:gamma-B} applied to \cref{thm:u-is-bialg} we conclude that $u$ is a morphism of Hopf algebras.  Applying these facts along with \cref{item:gamma-compat,item:theta-compat} to \cref{thm:r-is-coalg,thm:r-is-alg}, we conclude that $r$ is a morphism of Hopf algebras.  So we now consider the identities of \cref{sec:bialg-ids}.

  We consider the algebra identities first.  The assumptions give that \cref{eq:alg-inner} is equal to
    \begin{align}\label{eq:alg-inner-red}
      \delta_{g,xg'x\inv} \theta_g(x,x') \sum_{t\in A} u(e_{gt\inv})r(xx')\# p(e_t)v(xx'),
    \end{align}
  while \cref{eq:alg-outer} is equal to
    \begin{align*}
      \sum_{\substack{t,t'\in A\\b,b'\in B}} u(e_{gt\inv}) [v(x)\lact u(e_{g't^{\prime\inv}})] r(xx')p(t,b)p(t',b')\theta_j'(bv(x),b'v(x')) e_j \# bb' v(xx').
    \end{align*}
  Applying \cref{prop:thetaprime-bv-vanish} this reduces to
    \begin{align}\label{eq:alg-outer-red}
      \sum \theta_j'(v(x),v(x'))u(e_{gt\inv}) [v(x)\lact u(e_{g't^{\inv}})] r(xx') e_j \# p(e_t) v(xx').
    \end{align}
  For any fixed $j\in H$ and $g\in G$ in this sum, for the term to be non-zero by \cref{lem:u-simpler} we must have that $g,t$ satisfy $gt\inv = u^*(j)$, and so $t\in A$ is uniquely determined.  By \cref{item:theta-compat,item:omega-mod-A} we have
  \[ \theta_j'(v(x),v(x')) = \theta_{gt\inv}(x,x') = \theta_g(x,x').\]
  Since this value is necessarily non-zero, the equality of \cref{eq:alg-inner,eq:alg-outer} reduces to
  \begin{align*} \delta_{g,xg'x\inv} &\sum_{t\in A} u(e_{gt\inv})r(xx')\# p(e_t)v(x)\\
  =& \sum_{t\in A} u(e_{gt\inv}) [v(x).u(e_{g't^{\inv}})] r(xx') \# p(e_t) v(xx'),\end{align*}
  which is precisely the relation needed for $\morphquad$ to also define an algebra endomorphism of $\D(G)$.

  Now we look at the coalgebra identities.  Under the assumptions and applying \cref{prop:thetaprime-bv-vanish} we find that \cref{eq:coalg-outer} simplifies to
  \begin{align}\label{eq:coalg-outer-red}
    \sum_{\substack{j,l\in H\\t\in A\\b\in B}} \gamma'_{v(x)}(jl\inv,l)&u(gt\inv,j)r(x,j)p(t,b) e_{jl\inv}\# bv(x) \ot e_l \# bv(x)\nonumber\\
    =& \ \sum_{\substack{j,l\in H\\t\in A\\b\in B}} \gamma_x(gt\inv,t) u(gt\inv,j) r(x,j)p(t,b) e_{jl\inv}\# bv(x) \ot e_l\# bv(x)\nonumber\\
    =& \ \Delta_{\D(G)}\left( u(e_{gt\inv})r(x)\# p(e_t)v(x)\right)
  \end{align}
  On the other hand, \cref{eq:coalg-inner} simplifies to
  \begin{align*}
    \sum u(e_{gt\inv l\inv}) r(x)\# p(e_l)v(x) \ot u(e_{tk\inv})r(x)\# p(e_k)v(x) = \psi\ot\psi\circ\Delta_{\D(G)}(e_g\# x).
  \end{align*}
  These give precisely the relation needed for $\morphquad$ to define a coalgebra endomorphism of $\D(G)$.

  Combined, we conclude that $\morphquad$ is a Hopf algebra endomorphism of $\D(G)$, and it is bijective by default, and thus $\morphquad\in\Aut(\D(G))$.  Therefore by definition $\morphquad\in\genisomrg$, as desired.

  This completes the proof.
\end{proof}

As a corollary to the theorem, we see that not only can we typically expect $\genautr$ to be non-trivial, but that it typically has a non-trivial subgroup which is independent of the choice of 3-cocycle.

\begin{cor}\label{cor:bicharacters}
  The map $\BCh{G}\to \genautr$ given by $r\mapsto (0,1,r,1)$ is an injective group homomorphism.  As a consequence, there is an injective group homomorphism
  \begin{align*}
    \BCh{G}\to\bigcap_{\mathrlap{\omega\in Z^3(G,U(1))}}\Aut(\D^\omega(G)).
  \end{align*}
\end{cor}
\begin{proof}
  Apply \citep[Proposition 5.4]{K14} and the previous theorem to the special case where $\omega=\eta$, $p$ is trivial, and $u,v$ are identity maps.
\end{proof}
Recall that a group is said to be perfect if $G$ is equal to its derived subgroup $G'$, which is equivalent to $\widehat{G}\cong G/G'$ being trivial.  Since $\BCh{G}$ is equivalent to $\End(\widehat{G})$, which contains at least the trivial morphism and the identity morphism, we see that this subgroup is non-trivial if and only if $G$ is not perfect.

More generally, one may wonder if this inclusion is actually a bijection.  If $\morphquad$ is an element of the intersection which has $p$ non-trivial, then this would imply that there exists a non-trivial central subgroup $A$ of $G$ such that $G$ and $G/A$ have the same 3-cocycles: $Z^3(G,U(1))=Z^3(G/A,U(1))$.  On the other hand, if either $p$ or $r$ is trivial then it follows from \citep[Lemma 3.8]{K14} that $u,v$ are isomorphisms of Hopf algebras, and by \cref{lem:gen-u-omega-eta} we would conclude that $u^*$ acts trivially on $Z^3(G,U(1))$.  This leads to the following question.
\begin{question}\label{q:universal}
  Does there exist a finite group $G$ satisfying either of the following?
    \begin{enumerate}
      \item $G$ admits a non-trivial, proper, central subgroup $A$ such that \[Z^3(G/A,U(1))= Z^3(G,U(1))\] via the canonical projection $G\to G/A$.
      \item There exists $1\neq\alpha\in\Aut(G)$ such that $\alpha$ acts trivially on $Z^3(G,U(1))$.
    \end{enumerate}
\end{question}
If both are answered negatively---which seems probable to the author---then we are left to conclude that we indeed have a bijection, and that the intersection in question is non-trivial if and only if $G$ is perfect.  On the other hand, a positive answer to either implies there can be surprisingly many non-identity automorphisms which are defined independently of $\omega$.  The existence of any at all may already be a surprise.

Note that if a group satisfies the first condition, then $A$ is not contained in an abelian direct factor of $G$.  Subsequently $\alpha\in\Aut(G)$ acts trivially on $Z^3(G,U(1))$ whenever $\alpha(g)g\inv \in A$ for all $g\in G$.  Equivalently, the map $g\mapsto \alpha(g)g\inv$ is in $\Hom(G,A)$.  Conversely, given a non-trivial $\phi\in\Hom(G,A)$, then $g\mapsto \phi(g)g$ defines an element of $\Aut_c(G)$ which acts trivially on $Z^3(G/A,U(1))$.  As such we expect the first condition to likely be the more restrictive of the two.  Either one requires a complete description of all 3-cocycles, which is computationally intensive and apparently largely unconsidered in the literature.  After all, one is normally only interested in the cohomology class, so while there are well-known ways to compute representatives for every class there has been scant need to go beyond this.

Alternatively, one may view this as a consequence of the naive intersection being the wrong choice.  Since one is often interested in the representation categories $\Rep(\D^\omega(G))$, one should really consider some modification that will only depend on $H^3(G,U(1))$.  One such choice would be
\[ \bigcap_{[\gamma]\in H^3(G,U(1))} \bigcup_{\omega,\eta\in[\gamma]} \genisomr,\]
though one must be careful to note that this set is not obviously closed under composition.  This is the set of all $\psi\in\Aut(\D(G))$ such that for every $[\gamma]\in H^3(G,U(1))$ there are $\omega,\eta\in[\gamma]$ such that.  $\psi$ defines a rigid isomorphism $\D^\omega(G)\cong \D^\eta(H)$.  Perhaps more easily to say, it is the set of all element of $\Aut(\D(G))$ which induce tensor autoequivalences of $\Rep(\D^\omega(G))$ for all $[\omega]\in H^3(G,U(1))$, up to natural transformations.  Since $\Inn(G)$ acts trivially on (co)homology, this set contains $\Inn(G)\times\BCh{G}$, and the question of whether this is the entire intersection leads to the following.
\begin{question}\label{q:universal2}
Does there exist a finite group $G$ satisfying either of the following?
\begin{enumerate}
  \item $G$ admits a non-trivial, proper, central subgroup $A$ such that every (normalized) 3-cocycle of $G$ is cohomologous to a (normalized) 3-cocycle of $G/A$ (viewed as a 3-cocycle of $G$ via the canonical projection).  Equivalently, the canonical projection induces a surjection $H^3(G/A,U(1))\to H^3(G,U(1))$.
  \item There exists $1\neq \alpha\in\Out(G)$ such that $\alpha$ acts trivially on $H^3(G,U(1))$.
\end{enumerate}
\end{question}
As before, these conditions are related in the sense that if the first condition holds for same $A$, then if there exists a group homomorphism $\phi\colon G\to A$ such that $g\mapsto \phi(g)g$ is a non-trivial outer automorphism of $G$, then it necessarily acts trivially on $Z^3(G/A,U(1))$ and thus acts trivially on $H^3(G,U(1))$ by assumptions.

The second condition concerns the question of when those braided autoequivalences obtained from $\Aut(G)$ in the natural way, which are equivalent to the identity functor up to natural transformations, properly contains $\Inn(G)$.  On the other hand, the first condition concerns isomorphisms of the form $(p,1,0,1)$, and the tensor autoequivalences obtained from these need not be braided even when $\omega\equiv 1$, though there may exist braided equivalences which agree with them as tensor functors \citep{LenPri:Brauer}.

The second condition has been studied before \citep{JacMar:AutsTrivOnCohom,Coleman64,Hertweck01}, and the following result of \citet{Hertweck01} yields an example satisfying the second condition via the universal coefficient theorem and the divisibility of $U(1)$, though whether any group can satisfy the first is left open.

\begin{thm}[\citet{Hertweck01}]
  There exists a finite group $G$ with a non-trivial outer automorphism which is an inner automorphism of the integral group ring $\BZ G$.
\end{thm}

Now observe that the last two conditions of \cref{thm:rigid-isom-equiv} are trivially satisfied if $\eta\in Z^3(G/B,U(1))$.  Since $A,B$ are necessarily central by \cref{cor:inj-surj-cent}, applying \cite[Lemma 4.4]{K14} and \cref{thm:gen-omega-u} this statement and the second condition of the theorem are trivially satisfied if $\omega\in Z^3(G/Z(G),U(1))$ and $\omega^{u^*}=\eta$.

\begin{thm}\label{thm:quot-isoms}
  Suppose $\omega\in Z^3(G/Z(G),U(1))$.  Then $\morphquad\in\Aut(\D(G))$ satisfies $\morphquad\in\genisomrg$ if and only if $\omega^{u^*}=\eta$
\end{thm}
\begin{proof}
  The forward direction is \cref{thm:gen-omega-u}.  For the reverse direction, we continue the preceding discussion and observe that the only conditions of \cref{thm:rigid-isom-equiv} not satisfied are $\gamma_g(u^*(x),u^*(y)) = \gamma'_{v(g)}(x,y)$ for all $g\in G$ and $x,y\in H$, and $\theta_{u^*(x)}(g,h) = \theta'_x(v(g),v(h))$ for all $g,h\in G$ and $x\in H$.  But $\omega^{u^*}=\eta$ implies $\gamma'_x(y,z) = \gamma_{u^*(x)}(u^*(y),u^*(z))$ and $\theta'_x(y,z) = \theta'_{u^*(x)}(u^*(y),u^*(z))$.  It follows that these identities are equivalent to $\gamma_g(u^*(x),u^*(y)) = \gamma_{u^*v(g)}(u^*(m),u^*(n))$ for all $g,m,n\in G$ and $\theta_{u^*(x)}(g,h) = \theta_{u^*(x)}(u^*v(g),u^*v(h))$ for all $g,h,x\in G$.  Now by \cite[Lemma 4.4]{K14} we have $u^*v(g)g\inv \in Z(G)$ for all $g\in G$.  Since $\omega\in Z^3(G/Z(G),U(1))$ it now follows that these two identities are trivially true.  This completes the proof.
\end{proof}

The results of \citep{K14,KS14} suffice to give a description of all possible combinations of components that can appear in elements of $\Aut(\D(G))$ for any $G$; see also \citep{LenPri:Monoidal} for an alternative approach when $G$ is not purely non-abelian.  In this sense one can, at least in principle, completely determine $\genisomrg$ for any given $G$ and cocycles $\omega,\eta$.  The preceding theorem shows that the description becomes much simpler when $\omega\in Z^3(G/Z(G),U(1))$.  The next section considers a variety of examples of rigid isomorphisms and automorphisms.  The section after that will be dedicated to completely describing $\Aut(\D^\omega(G))$ whenever $\omega\in Z^3(G/Z(G),U(1))$.  This will generalize the case of $\omega\equiv 1$, and applies to all 3-cocycles whenever $Z(G)=1$, as special cases.

\section{Examples of Isomorphisms}\label{sec:rigid-examples}
Most of this section is dedicated to exhibiting rigid isomorphisms of twisted group doubles.  The final example of the section yields a non-rigid isomorphism of quasi-Hopf algebras.

Our first example is a simple and fairly obvious one, and uses \cref{thm:rigid-isom-equiv}
\begin{example}\label{ex:c2c2}
    Let $G = \BZ_2\times \BZ_2=\cyc{a}\times\cyc{b}$.  Let $\omega$ be a non-trivial 3-cocycle on $\BZ_2$, and let $\omega\times 1$ and $1\times\omega$ be the obvious coordinate liftings to $G$.  Let $\tau$ be the automorphism of $G$ that exchanges the two factors, and let $\sigma$ be the automorphism defined by $\sigma((a,1))=(a,1)$, $\sigma((1,b))=(a,b)$.  Let $p\colon\du{G}\to\kk G$ have $A=\cyc{a}$ and $B=\cyc{b}$.  Then we have rigid isomorphisms
    \begin{align*}
      (p,\tau^*,r,\tau)\colon \D^{1\times\omega}(G)\to \D^{\omega\times1}(G),
    \end{align*}
    where $r$ is any bicharacter of $G$.  Replacing $p$ by $p^*$ (effectively swapping the definitions of $A$ and $B$), yields rigid isomorphisms $\D^{\omega\times 1}(G)\to\D^{1\times\omega}(G)$, instead.  We also have rigid isomorphisms
    \begin{align*}
        (p,\tau^*,r,\sigma\tau)\colon \D^{1\times\omega}(G)\to \D^{\omega\times1}(G)\\
        (p^*,\tau^*,r,\tau\sigma)\colon\D^{\omega\times1}(G)\to\D^{1\times\omega}(G),
    \end{align*}
    with $p$ and $r$ as before.
\end{example}

We can also consider a cyclic group, for which we also use \cref{thm:rigid-isom-equiv}.

\begin{example}\label{ex:c6}
    Let $G=\BZ_6=\cyc{c}$.  The unique subgroup of order $3$ yields the canonical projection $G\to\BZ_2$, and we let $\omega$ be the (unique) non-trivial 3-cocycle of $\BZ_2$, which we view as a 3-cocycle on $G$ via the projection.  Let $v$ be the automorphism of $G$ which acts as inversion on elements of order 3 and by the identity on elements of order 2.  We also let $p$ have $A=B$ be the subgroup of order 3, and finally let $r$ be any bicharacter of $G$.  Then $(p,v^*,r,1)$, $(p,1,r,v)$, and $(p,v^*,r,v)$ are rigid automorphisms of $\D^\omega(G)$.
\end{example}

The next example considers a non-abelian group and a pair of non-cohomologous cocycles.  In this case, \cref{thm:quot-isoms} is applicable.
\begin{example}\label{ex:dihedral}
  Let $G= D_8 = \cyc{a, b \ : \ a^4 = b^2 = 1, \ bab=a\inv}$ be the dihedral group of order $8$.  We have $Z(G)=\cyc{a^2}\cong \BZ_2$ and $G/Z(G)=\cyc{\overline{a}}\times\cyc{\overline{b}} \cong \BZ_2^2$.  We define a 3-cocycle on $G/Z(G)$ by
  \[ \omega(x,y,z) = \begin{cases} -1 & \{x,y,z\}\subseteq \{\overline{a},\overline{b}\}\\ 1 & \text{otherwise} \end{cases}.\]
  By the canonical projection $G\to G/Z(G)$ we may identify this as a 3-cocycle on $G$.  The simple check that $\omega$ satisfies the 3-cocycle relation is left to the reader.  We define $v\in\Aut(G)$ by $v(a)=a$, $v(b)=ba$.  Since $Z(G)$ is a characteristic subgroup, $v$ also defines an element of $\Aut(G/Z(G))$ given by $gZ(G)\mapsto v(g) Z(G)$; indeed, this is a non-trivial automorphism of $G/Z(G)$.  In both cases, $v$ has order 2.  With $\omega$ viewed as a 3-cocycle on $G$, we define the 3-cocycle $\eta$ by $\eta=\omega^v$.  The theorem then tells us that $(0,v^*,0,v)\in \genisomrg$, and so $\D^\omega(G)\cong \D^\eta(G)$.  Since the order of $G$ is relatively small, whether or not $\omega,\eta$ are cohomologous is a system of equations that can be checked by brute force with a computer algebra system.  Using Mathematica \citep{Mathematica10}, the author verified that these 3-cocycles are not cohomologous, as desired.  Finally, it is trivial to verify that this isomorphism sends \[R=\sum_{g\in G} \varepsilon\# g \ot e_g\# 1\] to itself, and whence it is an isomorphism of quasi-triangular quasi-Hopf algebras.
\end{example}
One consequence of this is that the isomorphism constructed lifts to a braided equivalence of the fusion categories $\Rep(\D^\omega(D_8))$ and $\Rep(\D^\eta(D_8))$.  This provides an alternative proof to one of the tensor equivalences of twisted group doubles of dimension 64 found by \citet{GMN}.  Indeed, the existence of braided tensor equivalences $\Rep(\D^\omega(G))\to\Rep(\D^\eta(H))$  was completely characterized by \citet{NaNi}.

We note that constructing examples of \cref{thm:quot-isoms} with $u^*$ an isomorphism, and such that $\eta=\omega^{u^*}$ is not cohomologous to $\omega$, requires two conditions to be satisfied.  First, since $\Inn(G)$ acts trivially on the (co)homology of $G$, to yield non-cohomologous 3-cocycles the isomorphism must be a non-trivial outer automorphism.  Secondly, the isomorphism must be non-trivial on $G/Z(G)$, for else $\omega=\eta$.  As noted in the example, every element of $\Aut(G)$ defines an element of $\Aut(G/Z(G))$ in the obvious way since $Z(G)$ is characteristic.  So by definitions an element of $\Aut(G)$ acts trivially on $G/Z(G)$ if and only if it is a central automorphism.  Thus $u^*$ must be nontrivial in the quotient group \[ \frac{\Aut(G)}{\Inn(G)\Aut_c(G)}.\]

We conclude this section with a non-rigid isomorphism.  Our starting point is to look for an example which is "very close" to being rigid.  Morphisms take their simplest looking form when $p$ is trivial.  For isomorphisms, this forces $u,v$ to be isomorphisms of Hopf algebras by \cref{thm:inj-surj,thm:u-is-bialg,thm:v-is-bialg}.  By \cref{lem:coalg-to-group} the component $r$ also takes a nice form when it is a morphism of coalgebras.  So our natural guess of where to find an easy to understand non-rigid isomorphism is to look at those of the form $(0,u,r,v)$ where $r$ is a coalgebra morphism, but not an algebra morphism.  This can be arranged by picking a 3-cocycle satisfying \cref{thm:r-is-coalg} but not \cref{thm:r-is-alg}.

\begin{example}\label{ex:non-rigid}
    Let $G=\BZ_4=\cyc{c}$, $u=\id$, and $v$ the inversion map.  For any integer $m$ we define $[m]_4$ to be the (unique) smallest non-negative remainder of $m$ upon division by $4$.

    Define $r\colon \kk G\to\du{G}$ by setting $r(1)$ to be the trivial linear character of $G$, and $r(x)$ to be the linear character of order 2 when $x\neq 1$.  It follows from \cref{lem:coalg-to-group} that $r$ is a unital morphism of coalgebras, but is not an algebra morphism.

    Consider $\omega\in Z^3(G,U(1))$ defined by
    \[ \omega(c^l,c^j,c^k) = \exp\Big(\frac{\pi i}{4}[l]_4([j]_4+[k]_4-[j+k]_4)\Big).\]
    That this is actually a normalized 3-cocycle is both well-known and easily verified.  Note that $\omega$ is symmetric in the final two arguments, from which it follows that
    \[ \theta_g(x,y) = \omega(g,x,y) = \gamma_g(x,y).\]
    Moreover, the term $([j]_4+[k]_4-[j+k]_4)$ in the definition of $\omega$ is always an even integer, so $\omega$ takes only the values $\pm 1$.

    We claim that $(0,u,r,v)$ is a quasi-Hopf algebra automorphism of $\D^\omega(G)$, which is then necessarily not rigid by definition.

    We can write
    \[ (0,u,r,v)(e_{g}\# x) = \begin{cases} e_g\#x& x=1 \mbox{ or } g^2=1\\ - e_g\# x & \text{else} \end{cases}. \]
    It immediately follows that $(0,u,r,v)$ is bijective, and sends the coassociator to itself.

    We next have that $(0,u,r,v)$ is an algebra morphism if and only if
    \[ r(xy,j) \omega(j,x,y) = r(x,j)r(y,j)\omega(j,x\inv,y\inv)\]
    for all $x,y,j\in G$.  This can be verified to hold by a direct case-by-case check.  Note, in particular, that \cref{eq:theta-uv} fails to hold here.

    Now for $(0,u,r,v)$ to be a coalgebra morphism we must have
    \[ \omega(x,y,z)r(x,y)r(x,z) = \omega(x\inv,y,z)r(x\inv,yz).\]
    By definition $r$ is a morphism of coalgebras, so $r(x,y)r(x,z)=r(x,yz)$, and also by definition $r(x)=r(x\inv)$.  So $(0,u,r,v)$ is a coalgebra morphism if and only if $\omega(x,y,z)=\omega(x\inv,y,z)$ for all $x,y,z\in G$, which is easily verified to hold.

    It follows that $(0,u,r,v)$ is an isomorphism of quasi-bialgebras.  By \cref{lem:beta-pres} it also preserves the $\beta$ elements.  So, finally, using the observed properties of $\omega$ we note that the antipode is preserved if and only if $r(x,y)=r(x\inv,y\inv)$ for all $x,y\in G$. Since $r$ is a coalgebra morphism and takes real values we have $r(x\inv,y\inv) = r(x\inv,y)$.  By definition, $r(x\inv)=r(x)$ for all $x\in G$, and so we conclude that $(0,u,r,v)$ is an isomorphism of quasi-Hopf algebras, as desired.
\end{example}

\section{Rigid Automorphisms}\label{sec:rigid-auts}
Throughout this section we fix $\omega\in Z^3(G/Z(G),U(1))$.

From the action of $\Aut(G)$ on $Z^3(G,U(1))$ we can form the stabilizer subgroups
\begin{align*}
\Aut(G)_\omega = & \ \{v\in \Aut(G) \ : \ \omega^v=\omega\};\\
\Aut_c(G)_\omega = & \ \{v\in\Aut_c(G)\ : \ \omega^v=\omega\}.
\end{align*}
Note that our assumptions on $\omega$ guarantee $\Aut_c(G)_\omega=\Aut_c(G)$.

We recall the following group from \citep[Definition 5.6]{K14}:
\[ \SpAutc(G) = \{ (w,v)\in \Aut(G)\times\Aut(G) \ : \ w\inv v\in\Aut_c(G)\},\]
with group structure inherited from $\Aut(G)\times\Aut(G)$.  This is identified as a subgroup of $\Aut(\D(G))$ via \[ (w,v)\mapsto (0,(w\inv)^*,0,v).\]
Under this identification, by \cref{thm:quot-isoms} $(w,v)\in\SpAutc(G)$ also satisfies $(w,v)\in\Aut(\D^\omega(G))$ if and only if $\omega=\omega^w$.  Now $\omega\in Z^3(G/Z(G),U(1))$ and $w\inv v\in\Aut_c(G)$ yields that $\omega^w=\omega$ if and only if $\omega^v=\omega$.  So we may form the stabilizer subgroup
\begin{align*}
    \SpAutc(G)_\omega =& \{ (w,v)\in\Aut(G)_\omega\times\Aut(G)_\omega \ : \ w\inv v\in\Aut_c(G)\}\\
    \cong& \Aut_c(G)_\omega\rtimes \Aut(G)_\omega\\
    =&\Aut_c(G)\rtimes\Aut(G)_\omega
\end{align*}
which we identify as a subgroup of $\Aut(\D^\omega(G))$ as before:
\[ (w,v)\mapsto (0,(w\inv)^*, 0, v).\]  For the indicated isomorphism we have subgroups \[K=\{(v,v)\in\SpAutc(G)_\omega\}\cong \Aut(G)_\omega\] and
\[L=\{(w,1)\in\SpAutc(G)_\omega\}\cong\Aut_c(G)_\omega\] which give an exact factorization $\SpAutc(G)_\omega = LK$.  $K$ acts on $L$ by the usual conjugation action, which gives the desired semidirect product.

Finally we recall the following subgroup of $\Aut(\D(G))$ from \citep[Definition 5.1]{K14}:
\[ \Lambda(G) = \{(p,1,0,1)\in\Aut(\D(G))\}.\]
From \cref{thm:quot-isoms} we see this is also a subgroup of $\Aut(\D^\omega(G))$.  We then have the following.

\begin{thm}\label{thm:auts-center-quot}
  Let $\omega\in Z^3(G/Z(G),U(1))$, and define $\Lambda(G)$ and $\SpAutc(G)_\omega$ as above.  Then $\Aut(\D^\omega(G))$ contains the subgroup
  \[ \Lambda(G) \left( \SpAutc(G)_\omega\ltimes\BCh{G}\right). \]
  The indicated product of subgroups is an exact factorization of this subgroup.  Moreover, this subgroup contains all elements $\morphquad\in\Aut(\D^\omega(G))$ with either $u$ or $v$ invertible.

  If $G$ is purely non-abelian then this subgroup is all of $\Aut(\D^\omega(G))$, and so gives an exact factorization of $\Aut(\D^\omega(G))$.  On the other hand, if $G$ has non-trivial abelian direct factors, then the (non-twisted) reflections of abelian direct factors, as in \citep[Proposition 4.2(ii)]{LenPri:Monoidal}, give coset representatives of this subgroup.
\end{thm}
\begin{proof}
  Apply the preceding theorem and discussion along with \citep[Lemma 6.4]{K14} and \citep[Theorem 6.7]{K14} for the purely non-abelian case.  In the case when $G$ is not purely non-abelian, it is easily seen that the assumption on $\omega$ guarantees that reflections are in $\Aut(\D^\omega(G))$, and thus the coset decomposition for $\Aut(\D(G))$ given in \citep[Theorem 4.10]{LenPri:Monoidal} carries over to the subgroup $\Aut(\D^\omega(G))$, as desired.
\end{proof}

\bibliographystyle{plainnat}
\bibliography{../references}
\end{document}